\documentclass[a4paper, 12pt]{article}
\usepackage[usenames]{color} 
\usepackage{amssymb,amsthm,amsmath,amsfonts} 
\usepackage{lipsum}
\usepackage[a4paper]{geometry}
\usepackage[utf8]{inputenc} 
\usepackage[francais, english]{babel}
\usepackage{empheq}
\usepackage[T1]{fontenc}
\usepackage[all]{xy}
\usepackage{graphicx}

\theoremstyle{plain}
\newtheorem{theorem}{Theorem}[subsection]
\newtheorem{proposition}[theorem]{Proposition}
\newtheorem{corollary}[theorem]{Corollary}
\newtheorem{lemma}[theorem]{Lemma}
\newtheorem{definition-proposition}[theorem]{Definition/Proposition}
\theoremstyle{definition}
\newtheorem{definition}[theorem]{Definition}
\theoremstyle{remark}
\newtheorem{remark}[theorem]{Remark}
\theoremstyle{example}
\newtheorem{example}[theorem]{Example}
\theoremstyle{notation}
\newtheorem{notation}[theorem]{Notation}

\newcommand\C{\mathbb{C}}
\newcommand\R{\mathbb{R}}

\newcommand\h{{\text{\bfseries{H}}}}

\newcommand\esssup{{ess \ sup\ }}

\newcommand\zbar{\overline Z}
\newcommand\rphis{\Re(\phi_s)}
\newcommand\iphis{\Im(\phi_s)}
\newcommand\iphi{\Im(\phi)}
\newcommand\rpsis{\Re(\psi_s (S,X))}
\newcommand\ipsis{\Im(\psi_s (S,X))}
\newcommand\ipsi{\Im(\psi (S,X))}
\newcommand\rpsiss{\Re(\psi_{s,s} (S,X))}
\newcommand\ipsiss{\Im(\psi_{s,s} (S,X))}

\newcommand\iphiss{\Im(\phi_{s,s})}

\title{Geometric construction of quasiconformal mappings in the Heisenberg group}
\date{}
\author{Robin Timsit\\
\\
Institut de Mathématiques de Jussieu,\\
Université Pierre et Marie Curie,\\
4, place Jussieu,\\
75252 Paris,\\
France\\
e-mail : robin.timsit@imj-prg.fr}

\begin{document}

\maketitle
\abstract
In this paper, we are interested in the construction of quasiconformal mappings between domains of the Heisenberg group $\h$ that minimise a mean distortion functional. We propose to construct such mappings by considering a corresponding problem between domains of Poincaré half-plane $\mathbb H$. The first map we construct is a quasiconformal map between two cylinders. We explain the method used to find it and prove its uniqueness up to rotations. Then, we give geometric conditions for the construction to be the only way to find such minimizers. Eventually, as a non trivial example of the generalisation, we manage to reconstruct the map from \cite{BFP} between two spherical annuli.

\section*{Introduction and statement of results}
The theory of quasiconformal mappings in the complex plane is known to be a powerful tool to study deformations of complex structures. In spherical CR geometry, an adapted theory of quasiconformal mappings has been developped \cite {KR1, KR2} and used to define a distance in an analogue of Teichmüller space \cite {Wan}. In the case of spherical CR geometry, extremal quasiconformal mappings are still to be understood. Recently, some progress has been made in the area. A method using modulus of curve family has been developped \cite {BFP} in order to understand when a quasiconformal map has minimal mean distortion. In particular, the authors gave a condition, once we have a candidate for minimising a mean distortion, to verify if it is indeed a minimizer. However, finding a candidate for minimising a mean distortion seems tricky in general. Here, we are interested in the construction of such candidates. For other uses of modulus of curve family in CR geometry, we may quote \cite{Min, Kim} who studied quasiconformal conjugacy classes of CR-diffeomorphisms of the $3$-dimensional sphere.

In order to state our results, let us set notations and recall preliminary facts about the theory of quasiconformal mappings in the Heisenberg group. First, the Heisenberg group
$\h$ is the set $\C \times \R$ with the group law : if $(z,t), \ (z',t') \in \C \times \R$, then
\[ (z,t)\ast (z',t') = (z+ z' , t+t' +2 \Im (z\overline z')).\]
On $\h$, we have two left-invariant (complex) vector fields
\[ Z = \frac{ \partial}{\partial z} +i\overline z \frac{\partial}{\partial t} \text{ et } \overline Z = \frac{ \partial}{\partial \overline z} -iz \frac{\partial}{\partial t}.\]
If we set $T = \frac{\partial}{\partial t}$, one may verify that
\[ i[Z,\overline Z] = 2T.\]
The other commutator relations give zero. Noting $V$ the distribution $span (Z)$, $V$ is a CR structure on $\h$. It is known that the one point compactification of the Heisenberg group with this CR structure is CR-diffeomorhic to the $3$-dimensional sphere endowed with its standard CR structure. Thus, the Heisenberg group is a local model of spherical-CR geometry. Recall that a spherical CR-manifold is a $(G,X)$-manifold for $G = PU(2,1)$ and $X$ the three dimensional sphere.

A theory of quasiconformal mappings on the Heisenberg group was developped by Kor{\' a}nyi and Reimann, in what follows, we recall a few facts about it. For details, refer to \cite {KR1, KR2}. The Heisenberg group is endowed with a left invariant metric 
\[ d_{\text{\h}} (p,q) := \| p^{-1} \ast q \|_{\text{\h}}\]
where $\| (z,t) \|_{\h} := \left( |z|^4 + t^2 \right)^{\frac{1}{4}}$ is the Heisenberg norm. By analogy with the classical case, a homeomorphism $f : \Omega \longmapsto \Omega '$ between domains of $\h$ is called quasiconformal if 
\[ H(p,f) := \underset{r\to 0} {lim \ sup} \ \frac{\underset{d_{\text{\h}} (p,q) = r } {max}\  d_{\text{\h}} (f(p),f(q))}{\underset{d_{\text{\h}} (p,q) = r } {min}\  d_{\text{\h}} (f(p),f(q))}, \ p\in \Omega\]
is uniformly bounded. We say that $f$ is $K$-quasiconformal if $\|H(.,f)\|_{L^{\infty}} \le K$. As in the case of the complex plane, we have equivalent analytic definitions of quasiconformality. A sufficiently regular ($C^2$ is enough) quasiconformal map between domains of $\h$ has to be a contact map for the contact structure induced by the form $\omega = dt - i\overline z dz + i zd\overline z$, meaning that $f^* \omega = \lambda \omega$ for a nowhere vanishing real function $\lambda$. Moreover, by denoting $f = (f_1 , f_2 )$ with $f_1$ the complex part of the application and $f_2$ the real one, then, if $f$ is an orientation-preserving quasiconformal map, it satisfies a system of PDEs quite similar to Beltrami equation. Indeed, in that case, there is a complex valued function $\mu \in L^{\infty}$ (called Beltrami coefficient) with $\| \mu \|_{L^\infty} < 1$ such that
\[ \overline Z f_1 = \mu Z f_1 \text{ and } \overline Z \left( f_2 + i|f_1 |^2 \right) = \mu Z \left( f_2 + i|f_1 |^2 \right) \text{ a.e.. } \]
We then define the distortion function of the a map $f$ by
\[ K(p,f) := \frac{1 + |\mu (p)|}{1-|\mu (p)|} = \frac{ |Zf_1 (p) | + |\overline Z f_1 (p)|}{|Zf_1 (p) | - |\overline Z f_1 (p)|}\]
for $p \in \Omega$ where it makes sense and the maximal distortion of $f$ by $K_f := \underset{p\in \Omega} \esssup K(p,f)$. It is known that a conformal (i.e. $1$-quasiconformal) map $f : \Omega \longmapsto \Omega '$ is the restriction to $\Omega$ of the action of an element of $SU(2,1)$ (see \cite [p.~337]{KR1} for the smooth case and \cite[p.~869] {Cap} for the general one).

Here, we are interested in the following minimisation problem : consider a set of quasiconformal mappings $\mathcal F \subset \{ f : \Omega \longmapsto \Omega ' \ q.c. \}$. We are looking for a quasiconformal map $f_0 \in \mathcal F$ such that
\[ \int_\Omega K(p,f_0)^2\rho_0 ^4 dL^3 (p) = \underset{f\in \mathcal F} \min \int_\Omega K(p,f)^2\rho_0 ^4 dL^3 (p) \]
for a density $\rho_0$ depending on the geometry of the domain $\Omega$ and where $dL^3$ is the Lebesgue mesure on $\R^3$. When it is satisfied, we say that $f_0$ minimises the mean distortion on $\mathcal F$ for the density $\rho_0$. 

We propose here a geometric way to construct such minimisers in specific cases. The construction relies on the projection
\[\begin{array}{cccc}
\Pi : & \h \backslash \left(\{ 0 \} \times \R \right)& \longmapsto & \mathbb H\\
     & (z,t) & \longmapsto & t+ i|z|^2.
\end{array}\]
This projection comes from the CR identification between the Heisenberg group and the boundary of Siegel domain $E = \{ (z,w) \in \C^2 \ | \ \Im (w) > |z|^2 \}$ (that itself comes from the identification of standard CR structures of the one-point compactification of Heisenberg group and the three-dimensional sphere).  Usually, the boundary of Siegel domain is identified with $\C \times \R$ by $(z,w) \longmapsto (z, \Re (w))$. But, here we identify $\partial E \backslash \{ z=0 \}$ with a trivial circle bundle over the half plane $\mathbb H$ by $(z,w) \longmapsto \left( \frac{z}{\sqrt{\Im(w)}} , w \right)$. It gives a diffeomorphism 

\[\begin{array}{cccc}
\Psi^{-1} : & \h \backslash \left(\{ 0 \} \times \R\right) & \longmapsto & S^1 \times \mathbb H\\
     & (z,t) & \longmapsto & \left( \frac{z}{|z|} , t+ i|z|^2 \right)
\end{array}\]
and the projection, $\Pi$, is simply the second component of that diffeomorphism.

The idea is the following. Under appropriate geometric conditions on domains $\Omega$ and $\Omega '$ and on the density $\rho_0$, we can define a corresponding minimisation problem between two domains $U$ and $V$ of Poincaré half-plane. If we have a solution to the problem on $\mathbb H$, $g : U \longmapsto V$ such that there is a quasiconformal map $f = (f_1 , f_2 ) : \Omega \longmapsto \Omega '$ verifying $\left( f_2 + i |f_1 |^2 \right) (z,t) = g \left( t+i|z|^2 \right)$, then $f$ will be a solution to the problem on the Heisenberg group (Proposition 1.0.8. and Corollary 1.0.9.). We study more precisely an example between two cylinders. In that case, we manage to construct explicitly a unique (up to rotations) solution of the minimisation problem on $\h$ by lifting every solution of the corresponding problem on the half-plane, leading to Proposition 2.1.2. and Theorem 2.2.1.. Proposition 2.1.2. states that, in the case of the cylinder, there is only one solution of the corresponding problem in $\mathbb H$ that can be lifted by $\Pi$ into a quasiconformal map between cylinders. Theorem 2.2.1. states that a minimizer of the mean distortion functional considered between cylinders is inevitably the lift of a minimizer of the corresponding problem in $\mathbb H$. Then, we generalise the result obtained between cylinders to some domains of the Heisenberg group. Namely, under appropriate conditions on $\Omega$, $\Omega '$ and the density $\rho_0$, a minimizer $f$ has to be a lift by $\Pi$ of a minimiser for the corresponding problem in $\mathbb H$ (Theorem 3.2.4.). It reduces the problem of finding such a minimiser, to the resolution of an ordinary differential equation with boundary conditions (Proposition 3.1.2.).

We suppose, in the whole paper, that every quasiconformal map considered is $C^2$ and orientation preserving and every curve is $C^1$.

The paper is organized as follow. In Section 1, we present some theoretical background about moduli of curve families and state the problem we consider in the Heisenberg group and its corresponding one in the half-plane. Section 2 deals with construction and uniqueness (up to rotations) of a minimiser of a mean distortion between cylinders. We then generalise the construction in Section 3 and explain when it is the only way to find such minimisers; as an application, we reconstruct the extremal quasiconformal map between two spherical annuli found in \cite {BFP} and reduce the uniqueness problem to a boundary verification.

\section{Minimisation problem considered in \h \ and its corresponding one in $\mathbb H$}
\subsection* {Modulus of a curve family}

By analogy with the complex case, in order to understand extremal properties of a quasiconformal map between two domains of the Heisenberg group, we look at its behaviour on a well chosen family of curve that foliates the domain. We restrict the study here to $C^1$ curves and $C^2$ orientation preserving quasiconformal mappings, but most of the results of this section were proved in a general case. First of all, since a $C^2$ quasiconformal map is a contact transformation, we may restrict ourself to {\it horizontal} curves. 

\begin{definition} [{\bf Horizontal curves}]
A $C^1$ curve $\gamma : ]a,b[ \longmapsto \h$ is called horizontal if its tangents are in the contact distribution $D = ker (\omega)$. This condition is given explicitly by the following. Let $\gamma (s) = \left( \gamma_1 (s) , \gamma_2 (s) \right)$, $s \in ]a,b[$ be a curve in $\h$. Then, $\gamma $ is {\it horizontal} if and only if 
\[ \dot{\gamma}_2 (s) = -2 \Im (\overline \gamma_1 (s) \dot{\gamma}_1 (s)) \text{ for all $s\in ]a,b[$. } \]
\end{definition}

We can then define the modulus of a family of horizontal curves

\begin{definition}[{\bf Modulus of a family of horizontal curves}]

Let $\Gamma$ be a family of horizontal curves in a domain of $\h$, $\Omega$. We denote by $adm \left( \Gamma \right)$ the set of mesurable functions $\rho : \Omega \longmapsto [0, +\infty]$ such that $\int_{\gamma} \rho dl := \int_{a}^{b} \rho(\gamma (s))|\dot{\gamma}_1 (s)|ds \ge 1$ for all curves $\gamma \in \Gamma$. We call densities the elements of $adm(\Gamma)$. The modulus of the family $\Gamma$ is then defined by
\[M \left( \Gamma \right) := \underset{\rho\in adm(\Gamma)} \inf \int_{\Omega} \rho(p)^4 dL^3(p).\]

We say that a density $\rho_0$ is extremal if it verifies $M \left( \Gamma \right) =  \int_{\Omega} \rho_0(p)^4 dL^3(p)$. 

\end{definition}

When an extremal density exists, it is essentially unique (see Proposition 3.4. in \cite [p.~143]{BFP2}). There is a link between quasiconformality and modulus of a family of horizontal curves. Indeed, we have the following result proved in \cite [p.~177]{BFP}.

\begin{proposition}

Let $f : \Omega \longmapsto \Omega'$ be a quasiconformal map between domains of $\h$. Then, for every family of horizontal curves $\Gamma$ in $\Omega$ and every $\rho \in adm \left( \Gamma \right)$, one has
\[M(f(\Gamma)) \le \int_{\Omega} K(p,f)^2 \rho (p)^4 dL^3 (p).\]

\end{proposition}

Fixing a density $\rho$ and a $C^1$ quasiconformal map $f : \Omega \longmapsto \Omega '$, one may define a push-forward by $f$ of the density $\rho$.

\begin{definition-proposition} [{\bf Push-forward density}]

Let $f = (f_1 , f_2 ) :  \Omega \longmapsto \Omega '$ be a quasiconformal mapping between two domains of $\h$, $\Gamma$ a family of horizontal curves in $\Omega$ and $\rho \in adm (\Gamma)$.
Then, $\rho ' = \frac{\rho}{|Z f_1| - |\overline Z f_1 |} \circ f^{-1} \in adm \left( f \left( \Gamma \right) \right)$. Moreover,
\[ \int_{\Omega '} \rho '^4 dL^3 = \int_{\Omega} K(. , f)^2 \rho ^4 dL^3. \]
\end{definition-proposition} 

\begin{proof}

A simple application of the chain rule and the fact that every $\gamma \in \Gamma$ is a horizontal lead to the following. For every curve $\gamma = (\gamma_1 , \gamma_2 ) : ]a,b[ \longmapsto \Omega, \ \gamma \in \Gamma$, one has
\[\dot{\left( f_1 \circ \gamma \right)}(s) = Zf_1 (\gamma (s)) \dot \gamma_1 (s) + \overline Z f_1 (\gamma(s)) \dot {\overline \gamma}_1 (s) \text{ for every $s\in ]a,b[$. }\]
This leads to the important inequality : for every $s\in ]a,b[$,

\begin{eqnarray}
\left( |Zf_1 (\gamma (s)) | - |\overline Z f_1 (\gamma (s))| \right)\le \frac{|\dot {\left( f_1 \circ \gamma \right)}(s) |}{|\dot \gamma_1 (s) |} \le  \left( |Zf_1 (\gamma (s))| - |\overline Z f_1 (\gamma (s))| \right)
\end{eqnarray}
So, if $\gamma \in \Gamma$, using inequality (1) and the fact that $\rho \in adm (\Gamma)$, we find
\[\int_{f_1\circ \gamma} \rho' dl = \int_a ^b \frac{\rho}{|Zf_1 | - | \overline Z f_1 |} \circ \gamma (s)|\dot{(f_1 \circ \gamma)} (s)| ds \ge \int_a ^b \rho (\gamma (s)) |\dot \gamma_1 (s)| ds \ge 1. \]
For the second part, this is simply an application of the following change of variable formula for quasiconformal mappings (Theorem 16 in \cite[p.~175]{BFP}).
for every non-negative function $u : \Omega ' \longmapsto \R$, we have 
\[\int_{\Omega} (u\circ f) (p) |J(p,f)|dL^3 (p) = \int_{\Omega '} u(q)  dL^3 (q)\] 
where $J(p,f) = \left( |Zf_1 (p)| ^2 - |\overline Z f_1 (p) |^2 \right)^2$.
So, using this formula and the definition of $\rho '$, we have
\[ \int_{\Omega ' } \rho ' dL^3 = \int_{\Omega} \rho ^4 \frac{\left( |Zf_1|^2 - | \overline Z f_1 |^2 \right)^2}{\left( |Zf_1| - |\overline Z f_1 | \right)^4} dL^3 = \int_{\Omega} K(. , f)^2 \rho ^4 dL^3. \]

\end{proof}
Theorem 1 in \cite[p.~153]{BFP} gives a sufficient condition on the map $f$ to make the push-forward by $f$ of the extremal density of a family of curve $\Gamma $ extremal for the family $f(\Gamma)$.

\begin{theorem}

Let $\Omega$ and $\Omega '$ be bounded domains of $\h$. Let $\gamma : ]a,b[ \times \Lambda \longmapsto \Omega$ be a diffeomorphism that foliates $\Omega$, where $]a,b[ \subset \R$ with $a>0$ and $\Lambda$ is a domain of $\R^2$,such that $\gamma (\cdot , \lambda)$ is an horizontal curve verifying $|\dot{\gamma}_1 (s,\lambda) | \neq 0$ for all $\lambda \in \Lambda$ and $dL^3 (\gamma (s,\lambda)) = |\dot{\gamma}_1 (s,\lambda) |^4 ds d\mu(\lambda)$ for a mesure $d\mu$ on $\Lambda$.
\newline
\newline
Then, $\rho_0 (p) := \frac{1}{(b-a)|\dot{\gamma}_1 (\gamma^{-1} (p))|}$ is an extremal density for the family $\Gamma_0 := \{ \gamma (\cdot , \lambda) \ | \ \lambda \in \Lambda \}$.
\newline
Moreover, if $\mathcal F$ is a subset of the set of all quasiconformal map from $\Omega$ on $\Omega '$ and $f_0 \in \mathcal F$ is such that : 
\newline
1) $\mu_{f_0} (\gamma (s)) \frac{\dot{\overline{\gamma}}_1 (s)}{\dot{\gamma}_1 (s)} < 0$ for all $s\in ]a,b[$ 
\newline
2) For all $\lambda \in \Lambda$, $K (\gamma (s,\lambda), f_0)$ does not depend on $s$.
\newline
3) There is  $\Gamma \supset \Gamma_0$ such that $\rho_0 \in adm (\Gamma)$ and $M(f_0 (\Gamma_0)) \le M(f(\Gamma))$ for all $f\in \mathcal F$.
\newline
Then $f_0$ minimises the mean distortion on $\mathcal F$ for the extremal density $\rho_0$.

\end{theorem}

\subsection*{Statement of the corresponding problem in $\mathbb H$}

The previous theorem gives a way, once we have a candidate, to check if that candidate minimises a mean distortion functional. But, finding such a candidate may be quite challenging. Here, we explain how to construct such mappings in specific cases. As said in the introduction, the construction lies on the identification of $\h \backslash \{ z = 0 \}$ with $S^1 \times \mathbb H$ where $S^1$ is the unit circle of $\C$ and $\mathbb H$ is the Poincaré half plane, identification given by the diffeomorphism $\Psi^{-1}$ defined in the introduction whose inverse is the map

\[ \begin{array}{cccc}
\Psi :  & S^1 \times \mathbb H  & \longmapsto & \h \backslash \left(\{ 0 \} \times \R\right) \\
     & (e^{i\theta} ,w) & \longmapsto & \left( \sqrt{\Im(w)} e^{i\theta} , \Re(w) \right)
\end{array}\]

It gives new coordinates on $\h \backslash \left(\{ 0 \} \times \R\right)$ and a simple computation gives the following expression of vector fields $Z$ and $\overline Z$

\begin{eqnarray*} 
Z & = & 2i \sqrt{\Im(w)} e^{-i\theta} \partial_w - \frac{ie^{-i\theta}}{2\sqrt{\Im(w)}} \partial_\theta\\
\overline Z & = & -2i \sqrt{\Im(w)} e^{i\theta} \partial_{\overline w} + \frac{ie^{i\theta}}{2\sqrt{\Im(w)}} \partial_\theta.
\end{eqnarray*}

where $\partial_w = \frac{\partial}{\partial w}$, $\partial_{\overline w} = \frac{\partial}{\partial \overline w}$ and $\partial_\theta = \frac{\partial}{\partial \theta}$.
\newline
In the following, we consider $\widetilde \Omega$ and $\widetilde \Omega '$ domains in $\h \backslash (\{ 0 \} \times \R)$ such that $\Psi ^{-1} \left( \widetilde \Omega \right) = S^1 \times \Omega$ and $\Psi ^{-1} \left( \widetilde \Omega ' \right) = S^1 \times \Omega '$ with $\Omega$, $\Omega '$ domains of $\mathbb H$. We will look at lifts by $\Pi$ of curves in the half-plane. 

\begin{lemma}

Let $\gamma : ]a,b[ \longmapsto \mathbb H$ be a $C^1$ curve. Then, the only horizontal curves on $\h \backslash (\{ 0 \} \times \R)$, $\widetilde \gamma = (\gamma_1 , \gamma_2 )$ such that $ \Pi (\widetilde \gamma) = \gamma$ are the curves $\left( \sqrt{\Im (\gamma)} e^{i \tau} , \Re (\gamma) \right)$ where $\dot \tau = -\frac{\Re (\dot \gamma)}{2\Im(\gamma)}$.

\end{lemma}

\begin{proof}

Saying that $\gamma_2 + i|\gamma_1 |^2 = \gamma$ gives $\gamma_2 = \Re (\gamma)$ and $|\gamma_1 | = \sqrt{\Im(\gamma)}$. So, we only have to check that $\left( \sqrt{\Im(\gamma)} e^{i\tau} , \Re (\gamma) \right)$ is horizontal if and only if $\dot \tau = -\frac{\Re (\dot \gamma)}{2\Im(\gamma)}$, which is a simple application of the definition of a horizontal curve.

\end{proof}

Before going further, let's recall how the modulus of a curve family is defined in $\C$. Let $\Gamma$ be a family of curves $\gamma : ]a,b[ \longmapsto \Omega$ in a domain $\Omega$ of $\C$. We note again $adm(\Gamma)$ the set of mesurable functions $\rho : \Omega \longmapsto [0,\infty]$ such that $\int_{\gamma} \rho dl = \int_a ^b \rho (\gamma (s)) |\dot \gamma (s)| ds \ge 1$. The modulus of the family $\Gamma$ is 
\[ M(\Gamma) = \underset{ \rho \in adm(\Gamma)} \inf \int_{\Omega} \rho ^2 dL^2 \]
where $dL^2$ is the Lebesgue mesure of $R^2$. With that in mind, we can define the pull-back by $\Pi$ of a density.

\begin{definition-proposition}[{\bf Pull-back density}]

Let $\Omega$ be a domain in $\mathbb H$ and $\widetilde \Omega = \Psi (S^1 \times \Omega)$. Let $\Gamma$ be a curve family in $\Omega$ and note $\widetilde \Gamma$ its lifted family (defined by the previous lemma) in $\widetilde \Omega$. If $\rho \in adm(\Gamma)$, then $\widetilde \rho (z,t) := |Z\Pi (z,t)|\rho(\Pi (z,t)) = 2|z| \rho (t+ i |z|^2) \in adm (\widetilde \Gamma)$. We call the density $\widetilde \rho$ the pull-back by $\Pi$ of $\rho$.

\end{definition-proposition}

\begin{proof}

Let $\widetilde \gamma = (\gamma_1 , \gamma_2 ) \in \widetilde \Gamma$. By definition, there is $\gamma \in \Gamma$ such that $\gamma_2 + i |\gamma_1 |^2 = \gamma$. Since $\widetilde \gamma$ is horizontal, $\dot \gamma = 2i \overline \gamma_1 \dot \gamma_1$. Using the definition of $\widetilde \rho$, we find
\[\int_{\widetilde \gamma} \widetilde \rho dl = \int_a ^b 2|\gamma_1 (s)|\rho (\gamma (s)) |\dot \gamma_1 (s) | ds = \int_a ^b \rho(\gamma (s)) |\dot \gamma (s)| ds \ge 1.\]

\end{proof}

Let us note $\Pi ^\ast \rho$ the pull-back density by $\Pi$ and $f_\ast \widetilde \rho$ the push-forward density by $f$. One can define the same notion of push-forward density in $\C$ by setting $g_\ast \rho = \frac{\rho}{|\partial_w g|-|\partial_{\overline w} g|} \circ g^{-1}$ which satisfies
\[\int_{\Omega '} (g_\ast \rho)^2 dL^2 = \int_{\Omega} \rho^2 K(.,g) dL^2\]
where $K(.,g)$ is the quasiconformal distortion of $g$.

The following proposition and corollary explain the link between some mimisation problems in the Heinseberg group and corresponding problems in the half plane. 

\begin{proposition}

Let $\Omega$ and $\Omega '$ be domains of $\mathbb H$. We note $\widetilde \Omega = \Psi (S^1 \times \Omega)$ and $\widetilde \Omega ' = \Psi (S^1 \times \Omega ')$. Let $\Gamma$ be a curve family in $\Omega$ and $\widetilde \Gamma$ its lifted family in $\widetilde \Omega$. If $\rho \in adm (\Gamma)$ and $g : \Omega \longmapsto \Omega '$ is a quasiconformal map such that there is a quasiconformal map $f : \widetilde \Omega \longmapsto \widetilde \Omega '$ with $\Pi \circ f = g \circ \Pi$, then
\[ \Pi^\ast (g_\ast \rho ) = f_\ast (\Pi ^\ast \rho) \]

\end{proposition}

Before going through the proof, we give a corollary of this.

\begin{corollary}

Let $g : \Omega \longmapsto \Omega '$ and $f : \widetilde \Omega \longmapsto\widetilde \Omega '$ be as in the previous proposition. Let $\Gamma$ be a curve family in $\Omega$ with extremal density $\rho_0$. Suppose the following 
\newline
1) $M(\widetilde \Gamma) = \int_{\widetilde \Omega} (\Pi^{\ast} \rho_0)^4 dL^3$ where $\widetilde \Gamma$ is the lifted family of $\Gamma$,
\newline
2) $M(g(\Gamma)) = \int_{\Omega} \rho_0 ^2 K(.,g) dL^2$,
\newline
3) $M \left( \widetilde {g(\Gamma)} \right) = \int_{\Omega '} (\Pi ^\ast (g_\ast \rho_0))^4 dL^3$ where $\widetilde {g(\Gamma)}$ is the lift up family by $\Pi$ of $g(\Gamma)$.
\newline
Then, 
\[ M(f(\widetilde \Gamma)) = \int_{\widetilde \Omega} (\Pi^{\ast} \rho_0)^4 K(.,f)^2 dL^3.\]

\end{corollary}

\begin {proof} {[Proposition 1.0.8.]}
By definition, $g_\ast \rho = \frac{\rho}{|\partial_w g | - |\partial_{\overline w} g |} \circ g^{-1}$ and so
\[ \Pi^\ast (g_\ast \rho) = |Z \Pi| \left( \frac{\rho}{|\partial_w g |-  |\partial_{\overline w} g |} \circ g^{-1} \circ \Pi\right).\]
Moreover, since $f$ is contact, one as $Z (\Pi \circ f) = Z(f_2 + i|f_1|) = 2i \overline f_1 Zf_1$ and $\zbar (\Pi \circ f) = \overline Z (f_2 + i|f_1|) = 2i \overline f_1 \overline Z f_1$ (see \cite [p.~335]{KR1}). Thus, since $\Pi$ is a CR-function, using the chain rule we have

\begin{eqnarray*}
|Zf_1 | -|\zbar f_1| & = & \frac{1}{2|f_1|} \left( Z (\Pi \circ f) | - |\zbar (\Pi \circ f)| \right)\\
& = & \frac{1}{2|f_1|} \left( |Z(g\circ \Pi)| - |\zbar (g\circ \Pi)| \right)\\
& = & \frac{|Z\Pi|}{2|f_1|} \left( \left( |\partial_w g | - |\partial_{\overline w} g | \right) \circ \Pi \right)
\end{eqnarray*}

Now, computing $f_\ast (\Pi^\ast \rho)$, we find
\begin{eqnarray*}
f_\ast (\Pi^\ast \rho) & = & \frac{|Z\Pi| (\rho \circ \Pi)}{|Zf_1 | - |\zbar f_1|} \circ f^{-1}\\
& = & 2|f_1 \circ f^{-1} | \left( \frac{\rho}{  |\partial_w g | - |\partial_{\overline w} g | } \circ \Pi \circ f^{-1} \right)\\
& = & |Z\Pi| \left( \frac{\rho}{  |\partial_w g | - |\partial_{\overline w} g | } \circ g^{-1} \circ \Pi \right)
\end{eqnarray*}

\end{proof}

So, if we are looking for a quasiconformal map on the Heisenberg group that minimises the mean distortion for a nice density, a first step would be to look for the solutions of the corresponding problem in $\mathbb H$ that can be lifted by $\Pi$ into contact transformations. Which is what we will do in the next section for cylinders.

\section{Construction and uniqueness of an extremal quasiconformal map between cylinders}
\subsection{Construction of the map}

As said, we are looking for a quasiconformal map between cylinders defined as a lift by $\Pi$ of a quasiconformal map between projections of cylinders. We denote by $C_{r,R}$ the cylinder $\{ (z,t) \in \h \ | \ 0<t<r \text{ \& } |z|<\sqrt R \}$ for $r, R > 0$. We are here interested in finding a quasiconformal map $f : C_{a,b} \longmapsto C_{a' , b'}$ with $\frac{ab'}{a'b} > 1$ that minimises a mean distortion functional within the set $\mathcal F$ of all orientation preserving quasiconformal mappings from $C_{a,b}$ to $C_{a' , b'}$ that map homeomorphically the boundary components of $C_{a,b}$ into their corresponding boundary components in $C_{a', b'}$. 

Let $R_{a,b}$ and $R_{a', b'}$ be the rectangles $\{ w \in \mathbb H \ | \ 0< \Re(w) < a, \ 0<\Im(w) < b\}$ and $ \{ w \in \mathbb H \ | \ 0< \Re(w) < a', \ 0<\Im(w) < b'\}$. Then $\Pi (C_{a , b} \backslash \{ z=0\}) = R_{a,b}$ and $\Pi (C_{a' , b'} \backslash \{ z=0\}) = R_{a',b'}$. On the cylinder, $C_{a,b}$, there is a natural foliation by horizontal curves given by $\widetilde \Gamma_0 = \{ \widetilde \gamma_{z} (s) = \left( ze^{-\frac{is}{2|z|^2}},s \right) \ | \ 0<|z|<\sqrt b \text{ \& } 0<s<a \}$ which are the horizontal lifts by $\Pi$ of curves $\gamma _y (s) = s +iy$ for $0<y<b$ on $R_{a,b}$ (see Figure 1 next page).
\begin{figure}[!h]
\center
\includegraphics[width=7cm,height=7cm]{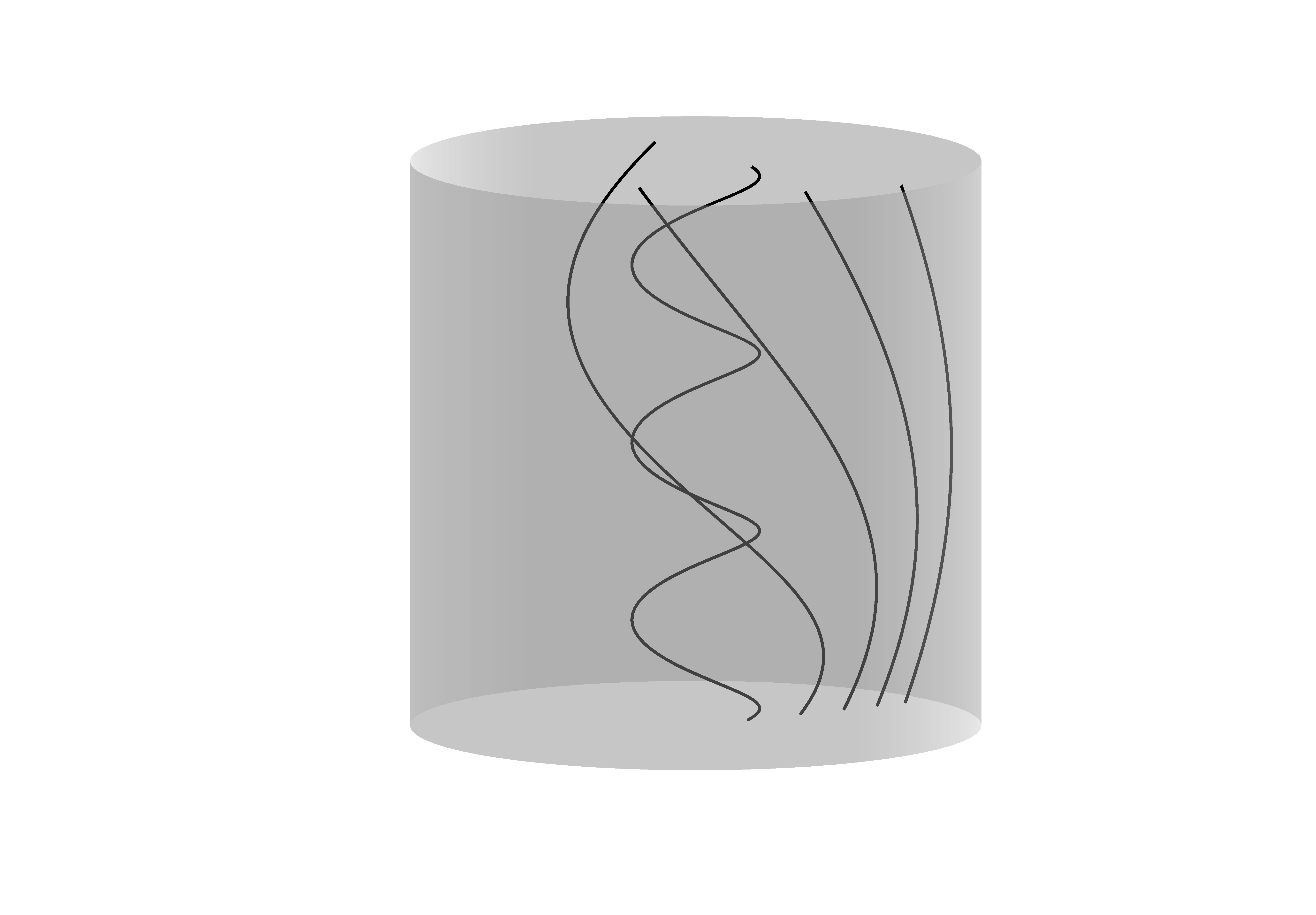}
\caption{Cylinder foliated by curves in $\widetilde \Gamma_0$ (foliation given by rotations around the vertical axis of drawn curves and the vertical axis itself).}
\end{figure}

To state the minimisation problem we are dealing with, we need to find the modulus and extremal density of  $\widetilde \Gamma_0$.

\begin{lemma}

The curve family $\widetilde \Gamma_0$ has modulus
\[ M(\widetilde \Gamma_0) = \frac{16\pi b^3}{3a^3}\]
and its extremal density is $\widetilde \rho_0 (z,t) = \frac{2|z|}{a}$.

\end{lemma}

\begin{proof}

If $\widetilde \rho \in adm(\widetilde \Gamma_0)$, by definition we have 
\[ \int_{\widetilde \gamma_z} \widetilde \rho dl = \int_0 ^a \widetilde \rho(\widetilde \gamma_z (s)) \frac{ds}{2|z|} \ge 1\]
for all $0<|z|<\sqrt b$.
So, 
\begin{eqnarray}
\int_0 ^a \widetilde \rho(\widetilde \gamma_z (s)) ds \ge 2|z|
\text{ for all $0<|z|<\sqrt b$ }
\end{eqnarray}
But,
\begin{eqnarray*}
\int_{C_{a,b}} \widetilde \rho ^4 dL^3 & \underset{\text{cylindrical coordinates}} = & \int_0 ^{2\pi} \int_0 ^{\sqrt b} \int_0 ^{a} \widetilde \rho^4 (r,\theta , t) r dt dr d\theta \\
 & \underset{\text{substitution with Jacobian 1}} = & \int _0 ^{2\pi} \int_0 ^{\sqrt b} \int_0 ^{a} \widetilde \rho^4 \left( r,\theta - \frac{s}{2r^2}, s \right) r ds dr d\theta\\
& = & \int _0 ^{2\pi} \int_0 ^{\sqrt b} \int_0 ^{a} \widetilde \rho^4 \left( \widetilde \gamma_{re^{i\theta}} (s) \right) r ds dr d\theta.
\end{eqnarray*}
Moreover,
\begin{eqnarray*}
\int_0 ^a \widetilde \rho^4 (\widetilde \gamma_{re^{i\theta}} (s)) ds & \underset{\text{Hölder inequality}} \ge &\frac{1}{a^3} \left( \int_0 ^a \widetilde \rho (\widetilde \gamma_{re^{i\theta}} (s)) ds \right)^4 \\
& \underset {\text{inequality (2)}} \ge & \frac{16r^4}{a^3}
\end{eqnarray*}
for all $0<r<\sqrt b$ and $0< \theta < 2\pi$.
Thus, we get
\[ \int_{C_{a,b}} \widetilde \rho^4 dL^3 \ge \frac{32\pi}{a^3} \int_0 ^{\sqrt b} r^5 dr = \frac{16\pi b^3}{3a^3}\]
for all $\widetilde \rho \in adm(\widetilde \Gamma_0)$. Consequently, $M(\widetilde \Gamma_0) \ge \frac{16\pi b^3}{3a^3}$. Moreover, set $\widetilde \rho_0 (z,t) = \frac{2|z|}{a}$. Then $\widetilde \rho_0 \in adm(\widetilde \Gamma_0)$ and $\int_{C_{a,b}} \widetilde \rho_{0} ^4 dL^3 = \frac{16\pi b^3}{3a^3}$.

\end{proof}

Let us set $\Gamma_0 = \{ \gamma_y \ | \ 0<y<b \}$. It is known since the work of Grötzsch that $\Gamma_0$ has modulus $\frac{b}{a}$ and extremal density $\rho_0 (w) = \frac{1}{a}$. Let $\mathcal G$ be the set of all quasiconformal mappings from $R_{a,b}$ on $R_{a' , b'}$ that map homeomorphically the boundary components of $R_{a,b}$ on their corresponding boundary components of $R_{a' , b'}$. The following is well known.

\begin{lemma}
Any minimizer of the mean distortion on $\mathcal G$ for the density $\rho_0$ is a map $f_{\varphi} (x + iy) = \frac{a'}{a} x + i\varphi (y)$ where $\varphi : [0,b] \longmapsto [0, b']$ is a function such that $\varphi(0) = 0$, $\varphi (b) = b'$ and $\dot \varphi \ge \frac{a'}{a}$.
\end{lemma}

Notice also that $\widetilde \rho_0 = \Pi^\ast \rho_0$. Then, according to Corollary 1.0.9., if we find a quasiconformal  map $\widetilde f \in \mathcal F$ that maps the vertical axis $\{ z = 0 \}$ homeomorphically on the vertical axis and such that $\Pi \circ \widetilde f = f_{\varphi} \circ \Pi$ for a certain function $\varphi$ defined as previously, then $\widetilde f$ will be a minimiser of the mean distortion on $\mathcal F$ for the density $\widetilde \rho_0$.

\begin{proposition}

There is only one function $\varphi : [0,b] \longmapsto [0, b']$ with $\varphi(0) = 0$, $\varphi (b) = b'$ and $\dot \varphi \ge \frac{a'}{a}$ such that $f_{\varphi}$ can be lifted into a quasiconformal map $\widetilde f : C_{a , b} \longmapsto C_{a' , b'}$. That function is defined by $\varphi (x) = \frac{b' x}{\left(1-\frac{ab'}{a'b}\right)x + \frac{ab'}{a'}}$ and the lifts are the rotations around the vertical axis of
\[\begin{array}{cccc}
\widetilde f_0 : & C_{a,b} & \longmapsto & C_{a', b'}\\
& (z,t) & \longmapsto & \left( \frac{\sqrt {b'} z e^{\frac{i}{2b} \left( 1 - \frac{a'b}{ab'} \right) t }}{\sqrt { \left( 1-\frac{ab'}{a'b} \right) |z|^2 + \frac{ab'}{a'}}} , \frac{a'}{a} t \right).
\end{array}\]

\end{proposition}

\begin{proof}
To prove this, it will be more convenient to write it in usual cylindrical coordinates on $\h = \R^3$. Meaning, $(r,\theta, t) \longmapsto (re^{i\theta} , t)$. In those coordinates, the contact form writes as $\omega = dt + 2r^2 d\theta$. So, we are looking for four functions $R, \Theta , T , \varphi $ such that $T+iR^2 (r,\theta , t) = \frac{a'}{a} t + i \varphi(r^2)$ and $dT + 2R^2 d\Theta = \lambda (dt + 2r^2 d\theta)$ for a nowhere vanishing function $\lambda$. Moreover, for all $r, t$, $\Theta (r, . , t)$ is $2\pi$-periodic modulo $2\pi$.  Since $T+iR^2 (r,\theta , t) = \frac{a'}{a} t + i \varphi(r^2)$, we get 
\[ T(r,\theta , t) = T(t) = \frac{a'}{a}t \text{ and } R^2 (r, \theta , t) = R^2 (r) = \varphi (r^2).\]
The idea is to use the system of PDEs that must verify the functions $R$, $\Theta$ and $T$ in order to find an ordinary differential equation that $\varphi$ must verify. In the following, we denote by an index $r$ (resp. $\theta$, resp. $t$) the partial derivative of a function according to $r$ (resp. $\theta$, resp. $t$).

Since $T_r = T_\theta = 0$, $T_t (t) = \frac{a'}{a}$, and  $dT + 2R^2 d\Theta = \lambda (dt + 2r^2 d\theta)$, we get that
\[ \Theta_r (r, \theta , t) = 0 \text{ and } \Theta_\theta (r, \theta , t) = \frac{a' r^2}{a\varphi (r^2)} + 2r^2 \Theta_t (r, \theta , t).\]
Moreover, from $\Theta_r = 0$, we deduce that $\Theta_{r, \theta} = \Theta_{r,t} = 0$. So, by deriving $\Theta_\theta (r, \theta , t) = \frac{a' r^2}{a\varphi (r^2)} + 2r^2 \Theta_t (r, \theta , t)$ according to $r$, we get
\[2\Theta_t (r,\theta , t) = \frac{a'}{a} \frac{r^2 \dot \varphi (r^2) - \varphi (r^2)}{\varphi ^2 (r^2)}\]
and deduce
\[ \Theta_{\theta} (r, \theta , t) = \Theta_{\theta} (r) = \frac{a'r^4 \dot \varphi (r^2)}{a\varphi ^2 (r^2)}. \]
From the fact that $\Theta_{r,\theta} = 0$, putting $x=r^2$, $\varphi$ must verify the differential equation 
\[ \frac{d}{dx} \left( \frac{x^2 \dot \varphi (x)}{\varphi ^2 (x)} \right) = 0\]
whose solutions are the functions  
\[ \varphi (x) = \frac{x}{Cx + D} \]
with $C,D \in \R$.
\newline
For such functions, we have $\Theta_\theta (r,\theta , t) = \frac{a'}{a} D$ and for $(R, \Theta , T)$ to be a homeomorphism, we must have $D = \frac{a}{a'}$. Moreover, we want that $\varphi (b) = b'$, so $C = \frac{1}{b'} \left( 1 - \frac{ab'}{a'b} \right)$. Consequently
\[ \varphi (x) = \frac{b' x}{\left(1 - \frac{ab'}{a'b}\right) x + \frac{ab'}{a'}}\]
for all $x\in [0,b]$ (one may check that $\left(1 - \frac{ab'}{a'b}\right) x + \frac{ab'}{a'} > 0$ if $x\in [0,b]$).
\newline
Replacing $\varphi$ by its value, we find
\begin{eqnarray*}
\Theta_\theta (r,\theta , t) & = & 1\\
\Theta_r (r,\theta , t) & = & 0 \\
\Theta_t (r,\theta , t) & = & \frac{1}{2b} - \frac{a'}{2ab'}.
\end{eqnarray*}
And so
\begin{eqnarray*}
\Theta (r,\theta , t) & = & \theta + \frac{1}{2}\left( \frac{1}{b} - \frac{a'}{ab'} \right)t + \alpha \text{ where $\alpha \in \R$}\\
R(r,\theta, t) & = & \frac{\sqrt {b'} r }{\sqrt{\left( 1- \frac{ab'}{a'b} \right)r^2 + \frac{ab'}{a'}}} \\ 
T (r,\theta , t) & = & \frac{a'}{a} t
\end{eqnarray*}
Which gives in usual coordinates
\[ \widetilde f_{\alpha} (z,t) = \left( \frac{\sqrt{b'} e^{i\alpha} z e^{\frac{i}{2b} \left( 1 - \frac{a'b}{ab'} \right) t }}{\sqrt { \left( 1-\frac{ab'}{a'b} \right) |z|^2 + \frac{ab'}{a'}}} , \frac{a'}{a} t \right). \]

\end{proof}
The mappings $\widetilde f_\alpha$ are quasiconformal with distortion function $K((z,t) , f_\alpha ) \\
= \frac{1}{\left( 1 + \left( \frac{a'}{ab'} - \frac{1}{b} \right)|z|^2 \right)^2}$ and maximal distortion $K_{\widetilde f_\alpha} = \left( \frac{ab'}{a'b} \right)^2$. Moreover, they map the vertical axis homeomorphically to the vertical axis, so according to what we said before the proposition, they minimise the mean distortion on $\mathcal F$ for the extremal density $\widetilde \rho_0$.

\subsection{Uniqueness up to rotations of the map}

Here, we are dealing with finding every quasiconformal mapping $f \in \mathcal F$ such that 
\[ M(f(\widetilde \Gamma_0)) = \int_{C_{a, b}} K(. , f)^2 \widetilde \rho_0 ^4 dL^3.\]
We will show that such maps must be constructed as we did in the previous section. So, for a map $f \in \mathcal F$ such that $M(f(\widetilde \Gamma_0)) = \int_{C_{a, b}} K(. , f)^2 \widetilde \rho_0 ^4 dL^3$, we only have to show one thing: $\Pi \circ f$ defines a map from $R_{a,b}$ to $R_{a' , b'}$, that is $\Pi \circ f (z,t)$ does not depend on $arg(z)$. Indeed, if we have that, Proposition 1.0.8. insures that  $\Pi \circ f$ defines a map that minimises the mean distortion on $\mathcal G$ for the extremal density $\rho_0$, and so $f$ must be defined as in the previous section.
Thus, the section is dedicated to the proof of the following.

\begin{theorem}

Suppose that $f \in \mathcal F$ verifies
\[ M(f(\widetilde \Gamma_0)) = \int_{C_{a, b}} K(. , f)^2 \widetilde \rho_0 ^4 dL^3.\]
Then, there is $\alpha \in \R$ such that $f= \widetilde f_{\alpha}$.

\end{theorem}

The proof is decomposed in three steps. The first two are a reformulation of the beginning of \cite {BFP2} in the setting of cylinders. In the third one, we finally prove that $\Pi \circ f$ does not depend on $arg(z)$.

We start by giving a caracterisation lemma for curves to be in $\widetilde \Gamma_0$

\begin{lemma}

Let $\Gamma$ be the set of all horizontal curves joining the two boundary discs of $C_{a,b}$ and take an element $\gamma$ de $\Gamma$. Then, 
\[ \int_{\gamma} \widetilde \rho_0 dl \ge 1.\] 
Moreover, we get equality if and only if $\gamma \in \Gamma_0$.

\end{lemma}

\begin{proof}
If $\gamma \in \Gamma$, take a parametrisation of $\gamma$ between $0$ and $a$, 
\begin{eqnarray*}
\int_{\gamma} \rho_0 dl & = & \frac{1}{a} \int_{0} ^{a} 2|\gamma _1 (s)| | \dot \gamma_1 (s)| ds\\
 & = & \frac{1}{a} \int_{0} ^{a} \left| \frac{d}{ds}(\gamma_2 + i |\gamma_1 |^2)(s) \right| ds\\
& \ge & \frac{1}{a} \int_{0} ^{a} |\dot \gamma_2 (s)| ds\\
& = & 1.
\end{eqnarray*}
Equality happens if and only if $|\gamma_1|^2$ is constant. In that case, one may check that $\gamma (s) = \left( ze^{-i \frac{\zeta (s)}{2|z|^2}} , \zeta (s) \right)$. Meaning, $\gamma \in \widetilde \Gamma_0$.

\end{proof}

Let $\widetilde \Gamma_0 '$ be the curve family $\{ \delta_z (s) = \left( ze^{-i\frac{s}{2|z|^2}} , s \right) \ | \ 
0<|z|<\sqrt {b'} \}$ we prove the following

\begin{proposition}

If $f$ is as in Theorem 2.2.1., then $f(\widetilde \Gamma_0 ) = \widetilde \Gamma_0 '$, meaning that for every $0<|z|<\sqrt b$, $f(\gamma _z (s)) = \left( z'e^{-i \frac{\zeta_z (s)}{2|z'|^2}} , \zeta_z (s) \right)$
where $\zeta_z$ is a homeomorphism from $]0,a[$ to $]0,a'[$. Moreover $f$ maps the vertical axis on the vertical axis.

\end{proposition}

\begin{proof}

The fact that $f$ minimises the mean distorition on $\mathcal F$ for the extremal density $\widetilde \rho_0$ insure that $f_\ast \widetilde \rho_0 $ is an extremal density of the family $\widetilde \Gamma_0 '$. But, $\widetilde \rho_0 ' (z,t) = \frac{2|z|}{a'}$ is also an extremal density of the family $\widetilde \Gamma_0 '$. Then, $f_\ast \widetilde \rho_0 = \widetilde \rho_0 '$. Let $\delta \in \widetilde \Gamma_0 '$, according to the previous lemma, we get,
\[ 1 = \int_{\delta } \widetilde \rho_0 ' dl = \int_{\delta } f_\ast \widetilde \rho_0 dl = \int_0 ^{a'} f_\ast \widetilde \rho_0  (\delta  (s)) |\dot \delta_1 (s)|ds.\]
Moreover, since $f^{-1} (\delta) \in \Gamma$, we have, using the fact that $|Z(f^{-1})_1| + |\zbar (f^{-1})_1| = \frac{1}{|Zf_1 | - |\zbar f_1|} \circ f^{-1}$ and inequality $(1)$,
\begin{eqnarray*}
1 & \le & \int_{f^{-1} (\delta)} \widetilde \rho_0 dl \\
& = & \int_0 ^{a'} \widetilde \rho_0 (f^{-1}(\delta (s) )) |\dot {((f^{-1})_1\circ \delta (s))}|ds \\
& \le & \int_0 ^{a'} \widetilde \rho_0 (f^{-1}(\delta (s) )) |\dot \delta (s)| \left(|Z(f^{-1})_1 (\delta (s))| + |\zbar (f^{-1})_1 (\delta (s))|\right) ds \\
& = & \int_0 ^{a'}  f_\ast \widetilde \rho_0  (\delta  (s)) |\dot \delta_1 (s)|ds\\
& = & 1
\end{eqnarray*}
Using the previous lemma, it means that $f^{-1} (\delta) \in \widetilde \Gamma_0$. So, for every $0<|z|< \sqrt {b'}$, $f^{-1} \left (ze^{-i\frac{s}{2|z|^2}} , s \right) = \left (R(z) e^{-i\frac{\zeta_z (s) }{2|R(z)|^2}} , \zeta_z (s) \right)$ for continuous functions $\zeta_z$ and $R$. Moreover, for every $t$, $(0, t) = \underset{z \to 0} \lim \left( ze^{-i\frac{t}{2|z|^2}} , t \right)$. Since $f$ is homeomorphic on the boundary, $\underset{z\to 0} |R(z)|$ is $0$ or $b$. If it were $b$, $f^{-1} (0,s)$ would be a horizontal curve which is a contradiction since $f$ maps horizontal curves to horizontal curves. So, $\underset{z\to 0} R(z) = 0$ and so $f^{-1}$ maps the vertical axis to the vertical axis. Thus, if we denote $\Gamma_0 ^\ast = \widetilde \Gamma_0 \cup \{s \mapsto (0,s)\}$ and $\Gamma_0 ^{' \ast} = \widetilde \Gamma_0 ' \cup \{s \mapsto (0,s)\}$, we have $f^{-1} ( \Gamma_0 ^{' \ast}) \subset \Gamma_0 ^\ast$. Since $\Gamma_0 ^\ast$ and $\Gamma_0 ^{' \ast}$ are foliations and $f$ is a homeomorphism, we get the result.

\end{proof}

Now we know that $f\left( ze^{-i\frac{s}{2|z|^2}} , s\right) = \left( z' e^{-i\frac{\zeta_z (s)}{2|z'|^2}} , \zeta_z (s) \right)$, we want to find the functions $\zeta_z$.

\begin{proposition}

For every $0<|z|<\sqrt b$ and $s\in ]0,a[$, we have
\[f\left( ze^{-i\frac{s}{2|z|^2}} , s\right) = \left( z' e^{-i\frac{a's}{2a|z'|^2}} ,  \frac{a'}{a} s \right)\]
for a complex number $0<|z'|<\sqrt{b'}$.

\end{proposition}

Before giving a proof, we need the following result : Proposition 2.12. in \cite[p.~133]{BFP2}. If $f$ is a map as in Theorem 2.2.1., then for every curve $\gamma \in \widetilde \Gamma_0$,
\begin{eqnarray}
|\dot{(f_1 \circ \gamma)}| & = & \left( |Zf_1 (\gamma) | - | \zbar f_1 (\gamma)| \right)|\dot \gamma|
\end{eqnarray}
This property is called the {\it minimal stretching property}.

\begin{proof}

In order to prove this, we consider two vector fields on $\h \backslash \{z=0\}$
\[ W := -\frac{iz}{2|z|^2}Z \text{ and } \overline W := \frac{i\overline z}{2|z|^2} \overline Z\]
Then, $W_{\gamma (s)} = \dot \gamma_1 (s) Z_{\gamma (s)}$ and $\overline W_{\gamma (s)} = \dot{\overline \gamma}_1 (s) \overline Z_{\gamma (s)}$ for every $\gamma \in \Gamma_0$. Thus,
\[|Wf_1 (\gamma) + \overline W f_1 (\gamma)| = |\dot \gamma_1 Zf_1(\gamma ) + \dot{\overline \gamma}_1 \overline Z f_1(\gamma)| = |\dot{(f_1 \circ \gamma)}|.\]
So, using (3), we have
\[|Wf_1 (\gamma) + \overline W f_1 (\gamma)| = \left(|Zf_1 (\gamma)| -|\zbar f_1 (\gamma)|\right) |\dot \gamma_1|.\] 
Since $|Zf_1 (\gamma)| -|\zbar f_1 (\gamma)| = \frac{1}{|\dot \gamma_1|} \left(|Wf_1 (\gamma)| - |\overline W f_1 (\gamma)|\right)$, we get that
\[|Wf_1 + \overline W f_1 | = |Wf_1| - |Wf_1|. \]
From this and the fact that $f$ is a contact transform, we also deduce
\[ |W(\Pi \circ f) + \overline W (\Pi \circ f)| = |W(\Pi \circ f)| - |\overline W (\Pi \circ f)| = 2|f_1| \left(|Wf_1| - |Wf_1| \right). \]
Finally, by definition of $W$ and $\overline W$, we find
\begin{eqnarray}
|W(\Pi \circ f)| - |\overline W (\Pi \circ f)| & = & \frac{|f_1|}{|z|} \left(|Zf_1 | - |\zbar f_1| \right)
\end{eqnarray}
Now, we know that 
\[f_\ast \widetilde \rho_0 \circ f (z,t) = \frac{\widetilde \rho_0}{|Zf_1| - |\zbar f_1|} \underset{by (4)} = \frac{2|f_1|}{a} \frac{1}{|W(\Pi \circ f)| - |\overline W (\Pi \circ f)|} =  \frac{2|f_1|}{a'}\]
Thus,
\[|W(\Pi \circ f)| - |\overline W (\Pi \circ f)| = \frac{a'}{a}\]
Moreover, Proposition 2.2.3. leads to $(\Pi \circ f \circ \gamma_z) (s) = \zeta_z (s) + i|z'|^2$ for every curve $\gamma_z \in \widetilde \Gamma_0$. So, $\dot \zeta_z (s) = \dot {(\Pi \circ f \circ \gamma_z)} (s)$. But,
\begin{eqnarray*}
|\dot {(\Pi \circ f \circ \gamma_z)}| & = & Z(\Pi \circ f) (\gamma_z) \dot \gamma_{z, 1} + \zbar (\Pi \circ f)(\gamma_z) \dot{\overline \gamma}_{z,1}|\\
& = & |W(\Pi \circ f)(\gamma_z) + \overline W (\Pi \circ f)(\gamma_z)|\\ 
& = & |W(\Pi \circ f)(\gamma_z)| - |\overline W (\Pi \circ f)(\gamma_z)|\\
& = & \frac{a'}{a}
\end{eqnarray*}
Thus, $\zeta_z (s) = \frac{a'}{a}s$ for every $0<|z|<\sqrt b$.
\end{proof}

In particular, we proved that $f_2 (z,t) = \frac{a'}{a} t$. Now, we are in position to show that $f_2 + i|f_1|^2$ does not depend on $arg(z)$. 

\begin{proof}{ [ Theorem 2.2.1.]}
As in the previous, it is more convenient to think in cylindrical coordinates. In those coordinates, the curves $\gamma_z$ are the curves $s \mapsto (r , \theta - \frac{s}{2r^2} , s)$ for $0<r<\sqrt b$ and $\theta \in \R$ and write the map $f$ as $(R,\Theta , T)$ (meaning that $(f_1 , f_2) = (Re^{i\Theta} , T)$). Since $f$ maps $\widetilde \Gamma_0$ to $\widetilde \Gamma_0 '$, then $\frac{d}{ds} R(r , \theta - \frac{s}{2r^2} , s) = 0$. Thus, $R_\theta (r , \theta - \frac{s}{2r^2} , s) = 2r^2 R_t (r , \theta - \frac{s}{2r^2} , s)$. As it is true for every $r, \theta$, we have 
\[R_\theta (r,\theta , t) = 2 r^2 R_t (r, \theta , t) \text{ for every $(r,\theta,t) \in ]0,\sqrt b[ \times \R \times ]0,a[$.}\]
By deriving according to $r$, we find also for every $(r,\theta,t) \in ]0,\sqrt b[ \times \R \times ]0,a[$,
\[R_{r, \theta} (r,\theta , t) = 4rR_t (r,\theta ,t) +2r^2 R_{r,t} (r,\theta ,t). \]
Since $(R,\Theta,T)$ is a contact map with $T_r (r,\theta , t) = 0$, then $\Theta_r (r, \theta , t) = 0$. Moreover, there is a nowhere vanishing function $\lambda$ such that for every $(r,\theta,t) \in ]0,\sqrt b[ \times \R \times ]0,a[$, we have
\[ \frac{a'}{a} + 2R^2 (r,\theta , t) \Theta_t (r,\theta , t) = \lambda (r,\theta , t) \text{ and } 2R^2 (r,\theta , t) \Theta_\theta (r,\theta , t) = 2r^2 \lambda (r,\theta ,t). \]
Leading to, for every $(r,\theta,t) \in ]0,\sqrt b[ \times \R \times ]0,a[$, 
\begin{eqnarray}
\Theta_\theta (r,\theta , t) & = & 2r^2 \Theta_t (r,\theta ,t) + \frac{a'r^2}{aR^2 (r,\theta , t)}.
\end{eqnarray}
Now, since $\Theta_r = 0$, by deriving the previous equation according to $r$, we get for every $(r,\theta,t) \in ]0,\sqrt b[ \times \R \times ]0,a[$,
\[ 0 = 4r\Theta_t(r,\theta ,t) + \frac{a'}{a} \left(\frac{2r}{R^2 (r,\theta , t)} - \frac{2r^2 R_r (r,\theta,t)}{R^3 (r,\theta ,t)} \right). \]
Then, for every $(r,\theta,t) \in ]0,\sqrt b[ \times \R \times ]0,a[$
\[ 2\Theta_t (r,\theta , t) = \frac{a'}{a} \left(\frac{r R_r (r,\theta,t)}{R^3 (r,\theta ,t)} - \frac{1}{R^2 (r,\theta , t)} \right).\]
Replacing in (5) : for every $(r,\theta,t) \in ]0,\sqrt b[ \times \R \times ]0,a[$
\[ \Theta_\theta (r, \theta , t) = \frac{a'r^3R_r (r,\theta , t)}{aR^3 (r,\theta , t)}.\]
By deriving the expression of $\Theta_t$ according to $\theta$, replacing $R_\theta (r,\theta , t)$ by $2r^2 R_t (r,\theta ,t)$ and $R_{r,\theta} (r,\theta , t)$ by $4rR_t (r,\theta ,t) +2r^2 R_{r,t} (r,\theta ,t)$, we find that for every $(r,\theta,t) \in ]0,\sqrt b[ \times \R \times ]0,a[$, 
\[\Theta_{\theta , t} (r,\theta , t) = \frac{a'r^2}{a} \left( \frac{4R_t (r,\theta , t) + r R_{t,r} (r,\theta , t)}{R^3 (r,\theta , t)} - \frac{3rR_r (r,\theta , t) R_t (r,\theta , t)}{R^4 (r,\theta , t)}\right). \]
By deriving the expression of $\Theta_\theta$ according to $t$, we have for every $(r,\theta,t) \in ]0,\sqrt b[ \times \R \times ]0,a[$,
\[\Theta_{t , \theta} (r,\theta , t) = \frac{a'r^2}{a} \left( \frac{ r R_{t,r} (r,\theta , t)}{R^3 (r,\theta , t)} - \frac{3rR_r (r,\theta , t) R_t (r,\theta , t)}{R^4 (r,\theta , t)}\right).\]
Since we assumed all our maps to be $C^2$, by use of Schwarz Theorem about commutativity of partial derivatives, we conclude that for every $(r,\theta,t) \in ]0,\sqrt b[ \times \R \times ]0,a[$
\[ \frac{4 R_t (r\theta, t)}{R^3 (r,\theta ,t)} = 0.\]
Leading to $R_t (r,\theta , t) = R_\theta (r,\theta, t) = 0$ for every $(r,\theta,t) \in ]0,\sqrt b[ \times \R \times ]0,a[$. Then, $f$ must be constructed as a lift up of a quasiconformal map between rectangles, in other words, as one of the $f_\alpha$. Which ends the proof of Theorem 2.2.1. 

\end{proof}

\section{Generalised construction}

In this section, we want to determine conditions in order to generalise the construction we made before to domains in $\h$ that are not conformally equivalent to cylinders but whose projections on the half plane are biholomorphic to rectangles. So, let us take two domains of $\mathbb H$, $\Omega_{a,b}$ and $\Omega_{a',b'}$ with two biholomorphic maps $\phi : R_{a,b} \longmapsto \Omega_{a,b}$ and $\psi : R_{a' , b'} \longmapsto \Omega_{a',b'}$ that extend homeomorphically to boundaries. 

\begin{notation}
We denote by $\Gamma_0$ (resp. $\Gamma_0 '$) the family of horizontal curves in $R_{a,b}$ (resp. $R_{a' , b'}$), $\Gamma_\phi = \phi (\Gamma_0)$ and $\Gamma_\psi = \psi (\Gamma_0 ')$. 
\newline
We denote also $\rho_0 = \frac{1}{a}$ (resp. $\rho_0 ' = \frac{1}{a'}$) the extremal density of $\Gamma_0$ (resp. $\Gamma_0 '$).
\newline
Finally, we denote $\rho_\phi = \phi_\ast \rho_0 \in adm(\Gamma_\phi)$ and $\rho_\psi = \psi_\ast \rho_0 ' \in adm(\Gamma_\psi)$ the push-forward densities. Since $\phi$ and $\psi$ are holomorphic mappings, $\rho_\phi$ (resp. $\rho_\psi$) is the extremal density of $\Gamma_\phi$ (resp. $\Gamma_\psi$).
\end{notation}

Recall from Lemma 2.1.2. that all quasiconformal mappings $f : R_{a,b} \longmapsto R_{a' , b'}$ sending homeomorphically boundary components on their corresponding ones in the target and such that $f_\ast \rho_0 = \rho_0 '$ are of the form $f_\varphi (x+iy) = \frac{a'}{a} x + i \varphi (y)$ for a function $\varphi : [0,b] \longmapsto [0,b']$ with $\varphi (0) = 0$, $\varphi (b) = b'$ and $\dot \varphi \ge \frac{a'}{a}$. Now, since the modulus of a curve family is a conformal invariant, every quasiconformal mapping $g : \Omega_{a,b} \longmapsto \Omega_{a' , b'}$ sending homeomorphically boundary components on their corresponding ones in the target space and such that $g_\ast \rho_\phi = \rho_\psi$ is one of the $g_\varphi = \psi \circ f_\varphi \circ \phi^{-1}$. A remark seems to be in order here to explain why the minimising problem between $\widetilde \Omega_{a,b}$ and $\widetilde \Omega_{a', b'}$ is, in theory, different from the one between cylinders.

\begin{remark}
Even though $\widetilde \Omega_{a,b}$ and $C_{a , b} \backslash \{z = 0 \}$ are diffeomorphic, there is absolutely no reason for them to be conformally homeomorphic. In fact, a map $\Phi : C_{a,b} \backslash \{z=0\} \longmapsto \widetilde \Omega_{a,b}$, such that  $\Pi \circ \Phi = \phi \circ \Pi$, is conformal if and only if $\phi$ is an element of $SL_{2} (\R)$. Moreover, if $\Phi$ is a conformal map, then it is minimal for the mean distortion; and we will see in Section 3.2 that in the case we consider, it implies that $\Phi$ defines a map $\phi$ such that $\Pi \circ \Phi = \phi \circ \Pi$ ($\phi$ here will be holomorphic because $\Phi$ is conformal). Thus, the problem of minimising the mean distortion between $\widetilde \Omega_{a,b}$ and $\widetilde \Omega_{a' , b'}$ must be handled another way than the one between cylinders.
\end{remark}

\subsection{Conditions for existence of a lift}
We wish here to find conditions on $\phi$, $\psi$ and $\varphi$ so that $g_\varphi$ can be lifted by $\Pi$ into a quasiconformal map between $\widetilde \Omega_{a,b} := \Psi (S^1 \times \Omega_{a,b})$ and  $\widetilde \Omega_{a',b'} := \Psi (S^1 \times \Omega_{a',b'})$. Namely, we will reduce the problem of finding a lift to the resolution of an ordinary differential equation. To do so, we make a change of coordinates in $\h$ more adapted to the problem. First, consider $R_{a,b}$ and $R_{a' , b'}$ as $]0,a[\times ]0,b[$ and $]0,a'[ \times ]0,b'[$ respectively and still write $\phi : ]0,a[\times ]0,b[ \longmapsto \Omega_{a,b}$ and $\psi : ]0,a'[\times ]0,b'[ \longmapsto \Omega_{a' , b'}$ the holomorphic maps. New coordinates are then given by the following two maps:
\[\begin{array}{cccc}
\Psi_{\phi} : & ]0,a[\times ]0,b[ \times \R & \longmapsto & \widetilde \Omega_{a,b}\\
& (s,x,\theta) & \longmapsto & \left( \sqrt{\Im(\phi (s,x))} e^{i\theta} , \Re (\phi (s,x)) \right)
\end{array}\]
\[\begin{array}{cccc}
\Psi_{\psi} : & ]0,a'[\times ]0,b'[ \times \R & \longmapsto & \widetilde \Omega_{a',b'}\\
& (s,x,\theta) & \longmapsto & \left( \sqrt{\Im(\psi (s,x))} e^{i\theta} , \Re (\psi (s,x)) \right)
\end{array}\]
So that $\Pi \circ \Psi_{\phi} (s,x,\theta)= \phi (s,x)$ and $\Pi \circ \Psi_{\psi} (s,x,\theta)= \psi (s,x)$ where it makes sense. Then, one may verify that a map $(S,X,\Theta) : ]0,a[\times ]0,b[ \times \R \longmapsto ]0,a'[\times ]0,b'[ \times \R$ defines a contact map from $\widetilde \Omega_{a,b}$ to $\widetilde \Omega_{a' , b'}$ if and only if there is a nowhere vanishing function $\lambda$ such that,
\begin{multline*}
\Re(\psi_s (S,X)) dS - \Im(\psi_s (S,X ))dX +2\Im(\psi (S,X))d\Theta  =  \\
\lambda\left(\Re(\phi_s ) ds -\Im(\phi_s)dx + 2\Im(\phi) d\theta \right).
\end{multline*}
Here again, we denote by an index $s$ (resp. $x$, resp. $\theta$) the partial derivative according to $s$ (resp. $x$ , resp. $\theta$).
We sum up this with a diagram. 

\[ \xymatrix {
]0,a[\times]0,b[\times \R  \ar[r]^{(S,X,\Theta)} \ar[d]_{\Psi_{\phi}} & ]0,a'[\times]0,b'[\times \R \ar[d]^{\Psi_{\psi}}\\
\widetilde \Omega_{a,b} \ar[d]_\Pi \ar[r]^{(g_1 , g_2)} & \widetilde \Omega_{a' , b'} \ar[d]^\Pi\\
\Omega_{a,b} \ar[r]_{g_\varphi} & \Omega_{a' , b'}
}\]

\begin{lemma}

Let $(S, X , \Theta) : ]0,a[\times]0,b[\times \R \longmapsto ]0,a'[\times]0,b'[\times \R$. The previous diagram is commutative if and only if, for every $(s,x,\theta) \in ]0,a[\times]0,b[\times \R$, we have 
\[ S(s,x , \theta) = S(s) = \frac{a'}{a} s \text{ and } X(s,x,\theta) = X(x) = \varphi (x).\]

\end{lemma}

\begin{proof}
The diagram is commutative if and only if we have $(g_\varphi \circ \Pi \circ \Psi_\phi) (s,x,\theta) = (\Pi \circ \Psi_\psi) (S(s,x,\theta),X(s,x,\theta),\Theta (s,x,\theta))$, for every $(s,x,\theta) \in ]0,a[\times]0,b[\times \R$. Leading to, for every $(s,x,\theta) \in ]0,a[\times]0,b[\times \R$,
\[(S, X) (s,x,\theta) = (\psi^{-1} \circ g_\varphi \circ \phi )(s,x) = f_\varphi (s,x) = \left( \frac{a'}{a} s , \varphi (x)\right)\]
\end{proof}

Now, we are able to state conditions for the existance of a contact lift.

\begin{proposition}

Suppose that $(S,X,\Theta) : ]0,a[\times]0,b[\times \R \longmapsto ]0,a'[\times]0,b'[\times \R$ is a contact transform such that $g_\varphi \circ \Pi \circ \Psi_\phi = \Pi \circ \Psi_\psi \circ (S,X, \Theta)$. Then, first,
$S(s,x,\theta) = \frac{a'}{a} s$ and $X(s,x, \theta) = \varphi (x)$ for every $(s,x,\theta) \in ]0,a[ \times ]0,b[ \times \R$. Moreover, $\varphi$, $\psi$ and $\phi$ satisfy for every $(s,x,\theta) \in ]0,a[ \times ]0,b[ \times \R$
\begin{eqnarray}
\frac{a'}{a} \dot \varphi (x) \frac{\left| \psi ' \left( \frac{a'}{a} s , \varphi (x)\right)\right|^2 \left(\Im (\phi (s,x))\right)^2}{|\phi ' (s,x)|^2 \left( \Im \left(\psi \left( \frac{a'}{a} s , \varphi (x)\right)\right)\right)^2} = 1.
\end{eqnarray}
Conversely, if those conditions are satisfied, then $(s,x,\theta) \longmapsto (\frac{a'}{a}s , \varphi (x) , \theta + h(s,x))$ with 
\begin{eqnarray*}
2h_x (s,x) & = & \frac{\Im \left( \psi_s \left( \frac{a'}{a} s , \varphi (x)\right) \right)}{\Im \left (\psi \left( \frac{a'}{a} s , \varphi (x)\right) \right)}  - \frac{\Im (\phi_s (s,x))}{\Im (\phi (s,x))}\\
2h_s (s,x) & = & \frac{\Re (\phi_s (s,x))}{\Im (\phi (s,x))} - \frac{a'}{a}\frac{\Re \left( \psi_s \left(  \frac{a'}{a} s , \varphi (x)\right) \right)}{\Im \left( \psi \left(  \frac{a'}{a} s , \varphi (x)\right) \right)}
\end{eqnarray*}
is a contact transform satisfying $g_\varphi \circ \Pi \circ \Psi_\phi = \Pi \circ \Psi_\psi \circ (S,X, \Theta)$.

\end{proposition}

\begin{proof}
The proof is similar to the one of Proposition 2.1.2., we take information from the map $(S,X,\Theta)$ to be contact in order to find the three partial derivatives of $\Theta$. According to the previous lemma, we know that $S(s,x,\theta) = \frac{a'}{a} s $ and $X(s,x, \theta) = \varphi (x)$ for every $(s,x,\theta) \in  ]0,a[\times]0,b[\times \R$. Now, the contact condition gives the following PDEs :
\begin{eqnarray}
\frac{a'}{a} \Re (\psi_s (S,X)) + 2\Im (\psi (S,X))\Theta_s & = & \lambda \Re (\phi_s)\\
\Im (\psi_s (S,X)) \dot \varphi - 2 \Im (\psi (S,X)) \Theta_x & = & \lambda \Im (\phi_s)\\
\Im (\psi (S,X)) \Theta_\theta & = & \lambda \Im (\phi) 
\end{eqnarray}
for a nowhere vanishing function $\lambda$. From (9), since $\Im (\phi) > 0$, we find $\lambda = \frac{\Im (\psi (S,X))}{\Im (\phi)} \Theta_\theta$. Replacing in (7) and (8) and dividing by $\Im (\psi (S,X)) > 0$, we deduce
\begin{eqnarray}
\frac{a'}{a}\frac {\Re (\psi_s (S,X))}{\Im (\psi (S,X))} + 2\Theta_s & = & \frac{\Re (\phi_s)}{\Im (\phi)} \Theta_\theta\\
\frac {\Im (\psi_s (S,X))}{\Im (\psi (S,X))}\dot \varphi (x) - 2\Theta_x & = & \frac{\Im (\phi_s)}{\Im (\phi)} \Theta_\theta.
\end{eqnarray}
Deriving (10) according to $x$, (11) according to $s$ and using Cauchy-Riemann equations for $\psi$ and $\phi$, we find
\begin{multline}
-\frac{a'}{a} \dot \varphi \left(\frac{\Im(\psi_{s,s} (S,X))}{\Im(\psi (S,X))} + \left(\frac{\Re(\psi_s (S,X))}{\Im (\psi (S,X))}\right)^2 \right) + 2\Theta_{x,s} = \frac{\Re(\phi_s)}{\Im(\phi)}\Theta_{x,\theta}\\
 -\left(\frac{\Im (\phi_{s,s})}{\Im (\phi)} + \left(\frac{\Re(\phi_s)}{\Im(\phi)}\right)^2\right)\Theta_\theta
\end{multline}
\begin{multline}
\frac{a'}{a} \dot \varphi \left(\frac{\Im(\psi_{s,s} (S,X))}{\Im(\psi (S,X))} - \left(\frac{\Im(\psi_s (S,X))}{\Im (\psi (S,X))}\right)^2 \right) - 2\Theta_{s,x} = \frac{\Im(\phi_s)}{\Im(\phi)}\Theta_{s,\theta}\\
 \left(\frac{\Im (\phi_{s,s})}{\Im (\phi)} + \left(\frac{\Im(\phi_s)}{\Im(\phi)}\right)^2\right)\Theta_\theta
\end{multline}
Thus, replacing the value of $2\Theta_{s,x}$ from (12) in (13), and using $|h'|^2 = (\Re(h_s))^2 + (\Im(h_s))^2$ for any holomorphic function, we find
\begin{eqnarray}
-\frac{a'}{a} \dot \varphi \frac{|\psi ' (S,X)|^2}{(\Im(\psi (S,X)))^2} - \frac{\Re(\phi_s)}{\Im(\phi)} \Theta_{x,\theta} + \frac{|\phi '|^2}{(\Im(\phi))^2} \Theta_\theta = \frac{\Im(\phi_s)}{\Im(\phi)} \Theta_{s,\theta}
\end{eqnarray}
Now, by deriving (10) and (11) both according to $\theta$, we also have $\Theta_{s,\theta} = \frac{\Re(\phi_s)}{2\Im(\phi)} \Theta_{\theta,\theta}$ and $\Theta_{x,\theta} = - \frac{\Im(\phi_s)}{2\Im(\phi)} \Theta_{\theta,\theta}$. So, replacing those in (14), we finally have
\[\Theta_\theta = \frac{a'}{a}\dot \varphi (x) \frac{\left| \psi ' \left( \frac{a'}{a} s , \varphi (x)\right)\right|^2 \left(\Im (\phi (s,x))\right)^2}{|\phi ' (s,x)|^2 \left( \Im \left(\psi \left( \frac{a'}{a} s , \varphi (x)\right)\right)\right)^2}.\]
The term on the right side does not depend on $\theta$. So $\Theta_{\theta , \theta} = 0 = \Theta_{s,\theta} = \Theta_{x,\theta}$. Thus, $\Theta_\theta$ is constant. So, for $\Theta$ to follow the periodicity condition, $\Theta_\theta$ must be everywhere equal to $1$. Which ends the first part of the proof. For the second part, it is a simple verification that a map $(S,X,\Theta)$ defined by $(s,x,\theta) \longmapsto (\frac{a'}{a}s , \varphi (x) , \theta + h(s,x))$ with 
$2h_x (s,x) = \frac{\Im \left( \psi_s \left( \frac{a'}{a} s , \varphi (x)\right) \right)}{\Im \left (\psi \left( \frac{a'}{a} s , \varphi (x)\right) \right)} \dot \varphi (x) - \frac{\Im (\phi_s (s,x))}{\Im (\phi (s,x))}$, 
$2h_s (s,x) = \frac{\Re (\phi_s (s,x))}{\Im (\phi (s,x))} - \frac{a'}{a}\frac{\Re \left( \psi_s \left(  \frac{a'}{a} s , \varphi (x)\right) \right)}{\Im \left( \psi \left(  \frac{a'}{a} s , \varphi (x)\right) \right)}$ and $\varphi$ satisfying (6) is a contact transform satisfying $g_\varphi \circ \Pi \circ \Psi_\phi = \Pi \circ \Psi_\psi \circ (S,X, \Theta)$ 

\end{proof}

\subsection{Geometric conditions for uniqueness of the construction}

In this section, our purpose will be to understand when minimisers of the mean distortion between $\widetilde \Omega_{a,b}$ and $\widetilde \Omega_{a' , b'}$ have to be lifts of minimisers of the mean distortion between $\Omega_{a,b}$ and $\Omega_{a' , b'}$. Let's make it more precise with some notations. 

\begin{notation}
First, we denote by $\widetilde \Gamma_\phi$ the family of horizontal lifts by $\Pi$ of $\Gamma_\phi$ and $\widetilde \Gamma_\psi$ the family of horizontal lifts by $\Pi$ of $\Gamma_\psi$. 
\newline
Then, we denote $\widetilde \rho_\phi$ the extremal density of $\widetilde \Gamma_\phi$ and $\widetilde \rho_\psi$ the extremal density of $\widetilde \Gamma_\psi$. 
\newline
Finally, we denote $\partial \Omega_{0,x} = \phi\left( \{0\} \times [0,b] \right)$, $\partial \Omega_{a,x} = \phi\left( \{a\} \times [0,b] \right)$, $\partial \Omega_{0',x} = \psi\left( \{0\} \times [0,b'] \right)$, $\partial \Omega_{a',x} = \psi\left( \{a'\} \times [0,b'] \right)$ and $\partial \widetilde \Omega_{\bullet,x} = \Pi ^{-1} (\partial \Omega_{\bullet,x})$ for $\bullet$ being $0, a, 0', a'$ respectively. The same way, denote by $\partial \Omega_{s,0} = \phi\left( [0,a]\times {0} \right)$, $\partial \Omega_{s,b} = \phi\left( [0,a] \times \{ b\} \right)$, $\partial \Omega_{s,0'} = \psi\left( [0,a'] \times \{0\} \right)$, $\partial \Omega_{s,b'} = \psi\left( [0,a'] \times \{b'\} \right)$ and $\partial \widetilde \Omega_{s,\bullet} = \Pi ^{-1} (\partial \Omega_{s,\bullet})$ for $\bullet$ being $0, b, 0', b'$ respectively. Here, $\Pi$ is to be understood as a function from $\h$ with value in $\mathbb H \cup \{ \Im (z) = 0 \}$.
\end{notation}

Let $\mathcal F_{\phi , \psi}$ be the set of all quasiconformal map from $f : \widetilde \Omega_{a,b} \longmapsto \widetilde \Omega_{a' , b'}$ such that $f(\partial \widetilde \Omega_{0, x}) = \partial \widetilde \Omega_{0',x}$, $f(\partial \widetilde \Omega_{a,x}) = \partial \widetilde \Omega_{a',x}$, $f(\partial \widetilde \Omega_{s, 0}) = \partial \widetilde \Omega_{s,0'}$ and $f(\partial \widetilde \Omega_{s,b}) = \partial \widetilde \Omega_{s,b'}$. We want to understand when a map $f \in \mathcal F_{\phi , \psi}$ such that $f_\ast \widetilde \rho_\phi = \widetilde \rho_\psi$ is a lift up of one of the $g_\varphi$. 
\newline
According to Corollary1.0.9., we already need that $\widetilde \rho_\phi = \Pi ^\ast \rho_\phi$ and $\widetilde \rho_\psi = \Pi ^\ast \rho_\psi$. Meaning that the extremal density of $\widetilde \Gamma_\phi$ (resp. $\widetilde \Gamma_\psi$) is exactly the pull-back by $\Pi$ of the extremal density of $\Gamma_\phi$ (resp. $\Gamma_\psi$). We give an exemple of a holomorphic map $\phi$ such that it is not the case.

\begin{example}

The example is quite simple. We know that in a rectangle $R_{a,b}$, horizontal lines satisfy the above property. But the vertical ones don't. Indeed, let $\Delta$ be the family of curves $\delta_x (s) = x + is$ for $0<x<a$ and $0<s<b$. It is well known that the modulus of the family $\Delta$ is $\frac{a}{b}$ with extremal density $\sigma = \frac{1}{b}$. Now, let $\widetilde \Delta$ be the family of horizontal lifts up of curves in $\Delta$. One may verify that $\widetilde \Delta$ is the family of curves $\widetilde \delta_{(z,t)} (s) = (sz , t)$ for $|z| = 1$, $0<t<a$ and $0<s<\sqrt b$. With a calculus quite similar to the one done in the proof of Lemma 2.1.1. , we find the modulus of $\widetilde \Delta$ to be $\frac{16\pi a}{27b}$ with extremal density $\widetilde \sigma (z,t) = \frac{2}{3b^{\frac{1}{3}} |z|^{\frac{1}{3}}}$ which is not the pull-back by $\Pi$ of $\sigma$. Now, we may send the rectangle $R_{b,a}$ to the rectangle $R_{a,b}$ by the composed of a rotation of angle $\frac{\pi}{2}$ and a translation by $a$ which is a holomorphic map sending horizontal lines to vertical ones. 

\end{example}

There is quite of a problem with the condition on densities: it is hard to find analytic consequences of it. But, we have a natural analytic condtion on maps $\phi$ and $\psi$ coming from equation (6). Indeed, equation (6) has a solution only if $\frac{\left| \psi ' \left( \frac{a'}{a} s , \varphi (x)\right)\right|^2 \left(\Im (\phi (s,x))\right)^2}{|\phi ' (s,x)|^2 \left( \Im \left(\psi \left( \frac{a'}{a} s , \varphi (x)\right)\right)\right)^2}$ is constant in $s$. A natural way to insure this, is to ask that $\frac{|\psi '|}{\Im(\psi)}$ and $\frac{|\phi '|}{\Im(\phi)}$ are both functions of $x$ only. Now, the following proposition is crucial to understand the geometry behind the uniqueness of th construction.

\begin{proposition}
If $\frac{|\phi '|}{\Im(\phi)}$ is a function of $x$ only, then 
\[ M(\widetilde \Gamma_\phi) = \int_{\widetilde \Omega_{a,b}} \Pi^\ast \rho_\phi dL^3. \] 
In other words, $\Pi^\ast \rho_\phi = \widetilde \rho_{\phi}$.
\newline
Conversely, if $M(\widetilde \Gamma_\phi) = \int_{\widetilde \Omega_{a,b}} \Pi^\ast \rho_\phi dL^3$, then $\frac{|\phi '|}{\Im(\phi)}$ is a function of $x$ only.
\end{proposition}

\begin{proof}

Let $\gamma_{x, \alpha} \in \widetilde \Gamma_\phi$. By definition, $\gamma_{x, \alpha} (s) = \left (\sqrt{\Im (\phi (s,x))} e^{i(\alpha + \tau (s,x))} , \Re (\phi (s,x)) \right)$ with $\tau (s,x) = -\int \frac{\Re(\phi_s (t,x))}{2\Im(\phi (t,x))} dt$ for every $s,x$. Then, computing, we have for every $s,x$
\[ |\dot \gamma_{x, \alpha,1} (s)| = \frac{|\phi ' (s,x)|}{2\sqrt{\Im(\phi(s,x))}} \]
where $\gamma_{x, \alpha, 1}$ is the first coordinate of the curve $\gamma_{x\alpha}$. Thus, if $\rho \in adm(\widetilde \Gamma_\phi)$, then, for every $x$ we have
\[ \int_0 ^a \rho (\gamma_{x, \alpha} (s)) \frac{|\phi ' (s,x)|}{2\sqrt{\Im(\phi(s,x))}} ds \ge 1. \]
But, by substitution, we have the following
\[\int_{\widetilde \Omega_{a,b}} \rho ^4 dL^3 = \frac{1}{2} \int_0 ^{2\pi} \int_0 ^b \int_0 ^a \rho^4 (\gamma_{x, \alpha} (s)) |\phi ' (s,x)|^2 dsdxd\alpha.\]
Moreover, using Hölder inequality, we have for every $x,\alpha$,
\begin{eqnarray*}
1& \le & \left(\int_0 ^a \rho (\gamma_{x, \alpha} (s)) \frac{|\phi ' (s,x)|}{2\sqrt{\Im(\phi(s,x))}} ds\right)^4\\
& \le & \int_0 ^a \rho^4 (\gamma_{x, \alpha} (s)) |\phi ' (s,x)|^2 ds \left(\int_0 ^a \left( \frac{|\phi ' (s,x)|^{\frac{1}{2}}}{2\sqrt{\Im(\phi (s,x))}} \right)^{\frac{4}{3}} ds\right)^3.
\end{eqnarray*}
By assumption, there is a function $h$ such that $h(x) = \frac{(\Im(\phi (s,x)))^2}{|\phi ' (s,x)|^2}$ for every $s,x$. Thus, we have for every $x$,
\[ \int_0 ^a \rho^4 (\gamma_{x, \alpha} (s)) |\phi ' (s,x)|^2 ds \ge \frac{16 h(x)}{a^3}. \]
Which leads to
\[ M(\widetilde \Gamma_\phi) \ge \frac{16\pi}{a^3} \int_0 ^b h(x) dx. \]
Now, by definition, $\Pi ^\ast \rho_\phi = \frac{|Z\Pi|}{a|\phi ' (\phi^{-1} \circ \Pi)|}$. So, again by substitution, 
\[ \int_{\widetilde \Omega_{a,b}} \left( \Pi ^\ast \rho_\phi \right)^4 dL^3 = \frac{1}{2} \int_0 ^{2\pi} \int_0 ^b \int_0 ^a \frac{16\left(\Im(\phi (s,x)) \right)^2}{a^4 |\phi ' (s,x) |^2} dsdxd\alpha = \frac{16\pi}{a^3} \int_0 ^b h(x) dx.\]

For the other side of the equivalence, using what was just done, we find for every $\rho \in adm(\widetilde \Gamma_\phi)$,
\[ \int_{\widetilde \Omega_{a,b}} \rho^4 dL^3 \ge 16\pi \int_0 ^b \frac{dx}{\left(\int_0 ^a \frac{|\phi '(s,x)|^{\frac{2}{3}}}{(\Im(\phi (s,x)))^{\frac{2}{3}}} ds \right)^3}.\]
Now, let $\rho\in adm (\widetilde \Gamma_\phi)$ be
\[ \rho := \left[ \left(\int_0 ^a \frac{|\phi '|^{\frac{2}{3}}}{2(\Im(\phi))^{\frac{2}{3}}} ds \right)^{-1} \frac{|\phi ' |^{-\frac{1}{3}}}{(\Im(\phi ))^{\frac{1}{6}}} \right] \circ \gamma^{-1}\]
where $\gamma (s,x,\alpha) = \gamma_{x,\alpha} (s)$.
Then, 
\[\int_{\widetilde \Omega_{a,b}} \rho^4 dL^3 = 16\pi \int_0 ^b \frac{dx}{\left(\int_0 ^a \frac{|\phi '(s,x)|^{\frac{2}{3}}}{(\Im(\phi (s,x)))^{\frac{2}{3}}} ds \right)^3}.\]
Since, $\Pi ^\ast \rho_\phi$ is extremal, we have $\Pi^\ast \rho_\phi = \rho$. This leads to
\[ \frac{|\phi '(s,x)|^{\frac{2}{3}}} {(\Im(\phi (s,x)))^{\frac{2}{3}}} = \frac{1}{a} \int_0 ^a \frac{|\phi '(s,x)|^{\frac{2}{3}}}{(\Im(\phi (s,x)))^{\frac{2}{3}}} ds\]
for every $s,x$. So, $\frac{|\phi '|}{\Im(\phi)}$ does not depend on $s$.
\end{proof}
The section is now dedicated to the proof of the following theorem, which may be understood as a converse of Corollary 1.0.9. in the case of domains $\Omega$ and $\Omega '$ biholomorphic to rectangles plus boundary conditions.

\begin{theorem}

Let $f : \widetilde \Omega_{a,b} \longmapsto \widetilde \Omega_{a' , b'}$ be a quasiconformal map in $\mathcal F_{\phi , \psi}$ such that $f_\ast \widetilde \rho_\phi = \widetilde \rho_\psi$ where $\widetilde \rho_\phi$ is the extremal density of $\widetilde \Gamma_\phi$ and $\widetilde \rho_\psi$ is the extremal density of $\widetilde \Gamma_\psi$. Suppose also that both densities are exactly the pull-backs by $\Pi$ of $\rho_\phi$ and $\rho_\psi$. Then, there is a quasiconformal map $g : \Omega_{a,b} \longmapsto \Omega_{a' , b'}$ sending homeomorphically the boundary components of $\Omega_{a,b}$ on the corresponding ones of $\Omega_{a', b'}$, such that $\Pi \circ f = g \circ \Pi$ and $g_\ast \rho_\phi = \rho_\psi$.

\end{theorem}

The proof follows the same steps as the one of Theorem 2.2.1.. Again the only thing we have to prove is that a quasiconformal map $f=(f_1,f_2)$ as in the previous theorem has the property: $(f_2 + i|f_1|^2) (z,t)$ does not depend on $arg(z)$. Let us set a quasiconformal map  $f : \widetilde \Omega_{a,b} \longmapsto \widetilde \Omega '$ with the hypothesis of the theorem. Then it fixes a map $(S,X,\Theta) : ]0,a[ \times ]0,b[ \times \R \longmapsto ]0,a'[ \times ]0,b'[ \times \R$ such that $\Psi_\psi \circ (S,X,\Theta) = f\circ \Psi_\phi$ and $\Re(\psi_s (S,X)) dS - \Im(\psi_s (S,X ))dX +2\Im(\psi (S,X))d\Theta  =  \lambda\left(\Re(\phi_s ) ds -\Im(\phi_s)dx + 2\Im(\phi) d\theta \right)$ for a nowhere vanishing function $\lambda$.

\begin{proposition}

Let $(S,X,\Theta) : ]0,a[\times ]0,b[ \times \R \longmapsto ]0,a'[\times ]0,b'[ \times \R$ be such a map. Assume moreover that it sends a curve $(s,x,\alpha + \tau (s,x))$ with $\tau_s (s,x) = -\frac{\Re (\phi_s (s,x))}{2\Im (\phi (s,x))}$ on a curve $\left( \frac{a'}{a} s , x' , \alpha ' + \upsilon \left(\frac{a'}{a} s , x'\right)\right)$ with $\upsilon_s (s,x) = -\frac{\Re (\psi_s (s,x))}{2\Im (\psi (s,x))}$. Then $X_\theta = 0$.

\end{proposition}
\begin{proof}

The proof is similar to the one of Theorem 2.2.1.. First, by hypothesis, $S(s,x,\theta) = \frac{a'}{a}s$ for every $(s,x,\theta) \in ]0,a[\times ]0,b[ \times \R$ and $\frac{d}{ds} X (s,x,\alpha - \tau (s,x)) = 0$. It leads to
\begin{eqnarray}
X_s = \frac{\Re (\phi_s)}{\Im(\phi)} X_\theta.
\end{eqnarray}
Now, since $(S,X,\Theta)$ defines a contact map, there is a nowhere vanishing function $\lambda$ such that
\begin{eqnarray}
\frac{a'}{a} \Re (\psi_s (S,X)) - \Im(\psi_s (S,X)) X_s + 2\Im (\psi (S,X))\Theta_s & = & \lambda \Re (\phi_s)\\
\Im (\psi_s (S,X)) X_x - 2 \Im (\psi (S,X)) \Theta_x & = & \lambda \Im (\phi_s)\\
-\Im(\psi_s (S,X)) X_\theta + 2\Im (\psi (S,X)) \Theta_\theta & = & 2\lambda \Im (\phi) .
\end{eqnarray}
From (18), since $\Im(\phi) > 0$, we find $\lambda = \frac{\Im(\psi (S,X))}{\Im(\phi)}\Theta_\theta - \frac{\Im(\psi_s (S,X))}{2\Im(\phi)}$. Replacing $\lambda$ by its value in (16) and (17), and using (15), we have the following
\begin{eqnarray}
\Theta_s & = & \frac{\Re(\phi_s)}{\Im(\phi)} \Theta_\theta - \frac{a'}{a} \frac{\Re(\psi_s (S,X))}{2\Im(\psi (S,X))} \\
\Theta_x & = & \frac{\Im(\psi_s (S,X))}{2\Im(\psi (S,X))} + \frac{\iphis}{2\iphi}\frac{\ipsis}{2\ipsi}X_\theta - \frac{\iphis}{2\iphi}\Theta_\theta.
\end{eqnarray}
By deriving (19) and (20) according to $\theta$, we find
\begin{eqnarray}
\Theta_{\theta , s} & = & \frac{\rphis}{2\iphi} \Theta_{\theta , \theta} + \frac{a'}{a} \left( \frac{\ipsiss}{2\ipsi} + \frac{(\rpsis)^2}{2(\ipsi)^2} \right)
\end{eqnarray}
\begin{multline}
\Theta_{\theta , x}  =  \left( \frac{\rpsiss}{2\ipsi} - \frac{\rpsis \ipsis}{2(\ipsi)^2} \right)X_\theta X_x \\ + \frac{\iphis}{2\iphi} \left(\frac{\rpsiss}{2\ipsi} - \frac{\rpsis \ipsis}{2(\ipsi)} \right) (X_\theta)^2 \\+ \frac{\ipsis}{2\ipsi}X_{\theta,x} + 
 \frac{\iphis}{2\iphi}\frac{\ipsis}{2\ipsi} X_{\theta , \theta} - \frac{\iphis}{2\iphi}\Theta_{\theta , \theta}.
\end{multline}
Moreover, deriving (15) according to $x$ and $\theta$ gives
\begin{eqnarray}
X_{x,s} & = & -\left( \frac{\iphiss}{2\iphi} + \frac{(\rphis)^2}{2(\iphi)^2} \right)X_\theta + \frac{\rphis}{2\iphi} X_{x,\theta} \\
X_{\theta, s} & = & \frac{\rphis}{2\iphi} X_{\theta , \theta}.
\end{eqnarray}
Now, we derive (19) according to $x$ and (20) according to $s$ (and replace $X_s$ by its value from (15) , $X_{s,x}$ by (23) , $X_{s,\theta}$ by (24) and $\Theta_{s,\theta}$ by (21) ).
\begin{multline}
\Theta_{x,s} = - \left( \frac{\iphiss}{2\iphi} + \frac{(\rphis)^2}{2(\iphi)^2} \right) \Theta_\theta + \frac{\rphis}{2\iphi} \Theta_{x,\theta}\\ + \frac{a'}{a} \left( \frac{\ipsiss}{2\ipsi} + \frac{(\rpsis)^2}{2(\ipsi)^2} \right)X_x
\end{multline}
\begin{multline}
\Theta_{s,x} = \frac{a'}{a} \left( \frac{\ipsiss}{2\ipsi} - \frac{(\ipsis)^2}{2(\ipsi)^2} \right) X_x \\
+ \frac{\rphis}{2\iphi}\left( \frac{\rpsiss}{2\ipsi} - \frac{\rpsis \ipsis}{2(\ipsi)^2} \right)X_{\theta,x}\\
- \frac{\ipsis}{2\ipsi} \frac{|\phi ' |^2}{2(\iphi)^2}X_\theta
+ \frac{\rphis}{2\iphi} \frac{\ipsis}{2\ipsi} X_{x,\theta}\\
+ \frac{a'}{a} \frac{\iphis}{2\iphi} \left(\frac{\ipsiss}{2\ipsi} - \frac{(\ipsis)^2}{2(\ipsi)^2} \right)X_\theta \\
\frac{\rphis \iphis}{4(\iphi)^2} \left( \frac{\rpsiss}{2\ipsi} - \frac{ \rpsis \ipsis}{2(\ipsi)^2} \right) (X_\theta)^2\\
+ \frac{\rphis \iphis}{4(\iphi)^2} \frac{\ipsis}{2\ipsi} X_{\theta,\theta} + \left( \frac{\iphiss}{2\iphi} + \frac{(\iphis)^2}{2(\iphi)^2} \right) \Theta_\theta\\
- \frac{\rphis \iphis}{4(\iphi)^2} \Theta_{\theta , \theta} - \frac{a'}{a} \frac{\iphis}{2\iphi} \left( \frac{\ipsiss}{2\ipsi} + \frac{(\rpsis)^2}{2(\ipsi)^2} \right) X_{\theta}.
\end{multline}
Now, since $\Theta_{s,x} = \Theta_{x,s}$ the previous two are equal. Thus, using the value of $\Theta_{x,\theta}$ found in (22), we get
\begin{eqnarray}
\Theta_\theta & = & \frac{a'}{a} \frac{|\psi ' (S,X) |^2 (\iphi)^2}{|\phi ' |^2 (\ipsi)^2} \left( X_x + \frac{\iphis}{2\iphi} X_\theta \right) + \frac{\ipsis}{2\ipsi} X_\theta.
\end{eqnarray}
Replacing in (19) we have
\begin{multline}
\Theta_s = \frac{a'}{a} \frac{\rphis}{2\iphi}\frac{|\psi ' (S,X) |^2 (\iphi)^2}{|\phi ' |^2 (\ipsi)^2} \left( X_x + \frac{\iphis}{2\iphi} X_\theta \right)\\ + \frac{\rphis}{2\iphi}\frac{\ipsis}{2\ipsi} X_\theta - \frac{a'}{a} \frac{\rpsis}{2\ipsi}
\end{multline}
By assumption, we can write $\frac{|\psi ' (S,X) |^2 (\iphi)^2}{|\phi ' |^2 (\ipsi)^2} = h(x,X)$ for a real valued function $h : ]0,b[ \times ]0,b'[ \longmapsto \R$. We write $h_2$ the partial derivative of $h$ according to the second variable. Now, deriving (27) according to $s$ and (28) according to $\theta$, and using formulae for $X_{s}$ and $X_{s,\theta}$, we find
\begin{multline}
\Theta_{s,\theta} = \frac{a'}{a}h_2 (. , X)X_\theta \left(\frac{\rphis}{2\iphi}X_x + \frac{\rphis \iphis}{4(\iphi)^2} X_\theta \right) \\
+ \frac{a'}{a} h(.,X) \left( \frac{\rphis}{2\iphi}X_{x,\theta} - \frac{|\phi '|^2}{2(\iphi)^2} + \frac{\rphis \iphis}{4(\iphi)^2}X_{\theta , \theta} \right) \\
+ \frac{\rphis}{2\iphi} \left( \frac{\rpsiss}{2\ipsi} - \frac{\rpsis}{\ipsis}{2(\ipsi)^2} \right) (X_\theta)^2 \\
+ \frac{\rphis}{2\iphi}\frac{\ipsis}{2\ipsi} X_{\theta, \theta}
\end{multline}
\begin{multline}
\Theta_{\theta , s} = \frac{a'}{a}h_2 (. , X)X_\theta \left(\frac{\rphis}{2\iphi}X_x + \frac{\rphis \iphis}{4(\iphi)^2} X_\theta \right)\\
+ \frac{\rphis}{2\iphi} \left( \frac{\rpsiss}{2\ipsi} - \frac{\rpsis \ipsis}{2(\ipsi)^2} \right) (X_\theta)^2 \\
+ \frac{a'}{a} \frac{\rphis}{2\iphi} \left(X_{x, \theta} + \frac{\iphis}{2\iphi}X_{\theta , \theta} \right) + \frac{\rphis}{2\iphi}\frac{\ipsis}{2\ipsi} X_{\theta , \theta}\\ + \frac{a'}{a} \left( \frac{\ipsiss}{2\ipsi} + \frac{(\rpsis)^2}{2(\ipsi)^2} \right)X_\theta.
\end{multline}
Finally, using the fact that $\Theta_{s,\theta} = \Theta_{\theta , s}$ and the definition of $h$, we find $\frac{2|\psi ' (S,X)|^2}{(\Im (\psi (S,X)))^2} X_\theta = 0$ which leads to $X_\theta = 0$.

\end{proof}

Our purpose now is to show that a minimizer of the mean distortion in $\mathcal F_{\phi , \psi}$ must be defined by a map $(S,X,\Theta)$ that sends a curve $(s,x,\alpha + \tau (s,x))$ where $\tau_s (s,x) = -\frac{\Re (\phi_s (s,x))}{2\Im (\phi (s,x))}$ on a curve $\left( \frac{a'}{a} s , x' , \alpha ' + \upsilon \left(\frac{a'}{a} s , x'\right)\right)$ with $\upsilon_s (s,x) = -\frac{\Re (\psi_s (s,x))}{2\Im (\psi (s,x))}$. For that, we will follow essentially what we made in section 2.2.. First, a curve $\widetilde \gamma (t) = (s(t) , x(t) , \theta (t))$ in $]0,a[ \times ]0,b[ \times \R$ (resp. in $]0,a'[ \times ]0,b'[ \times \R$) is said to be horizontal if $\Psi_\phi (\widetilde \gamma (t))$ is horizontal in $\widetilde \Omega_{a,b}$ (resp. $\Psi_\psi (\widetilde \gamma (t))$ is horizontal in $\widetilde \Omega_{a',b'}$).

\begin{lemma}

Let $\widetilde \Gamma$ be the family of all horizontal curves $\widetilde \gamma (t) = (s(t) , x(t) , \theta (t))$ such that $s(0) = 0$ and $s(a) = a'$. Then, for every $\widetilde \gamma \in \widetilde \Gamma$, 
\[ \int_{\Psi_{\phi (\widetilde \gamma)}} \widetilde \rho_\phi dl \ge 1. \]
Moreover, we have equality if and only if $\Psi_\phi (\widetilde \gamma) \in \widetilde \Gamma_\phi$.

\end{lemma}

\begin{proof}

Let $\widetilde \gamma$ be a curve in $\widetilde \Gamma$ and $\gamma = \Psi_\phi (\widetilde \gamma)$. Then, since, according to Proposition 3.2.2., $\widetilde \rho_\phi = \Pi ^\ast \rho_\phi$, we have the following
\begin{eqnarray*}
\int_\gamma \widetilde \rho_\phi dl & = & \int_0 ^a \frac{2|\gamma_1 (t)||\dot \gamma_1(t)|}{a|\phi ' (\phi^{-1}(\Pi \circ \gamma)(t)) |} dt\\
& = & \frac{1}{a}\int_0 ^a \frac{\dot{(\Pi \circ \gamma)}(t)}{|\phi ' (\phi^{-1}(\Pi \circ \gamma)(t)) |} dt\\
& = & \frac{1}{a}\int_0 ^a |\dot s (t) + i\dot x(t)|dt\\
& \ge & \frac{1}{a}\int_0 ^a |\dot s (t) | dt\\
& = & 1 
\end{eqnarray*}
We have equality here if and only if $\dot x = 0$. One may then verify that a curve $(s(t) , x_0 , \theta (t))$ is horizontal if and only if its image by $\Psi_\phi$ is an element of $\widetilde \Gamma_\phi$.

\end{proof}
Now, we may prove the following the same way that we proved Proposition 2.2.3..
\begin{proposition}

In our setting we have
\[ f(\widetilde \Gamma_\phi) = \widetilde \Gamma_\psi.\]

\end{proposition}

Thus, according to this proposition, the map $(S,X,\Theta)$ sends a curve $\widetilde \gamma_{(x,\alpha)} =  (s,x, \alpha + \tau (s,x))$ with $\dot \tau (s,x) = - \frac{\Re(\phi_s (s,x))}{2\Im(\phi (s,x))}$ on a curve $\widetilde \delta_{(x' , \alpha ')} (s) = (\zeta_{(x,\alpha)} (s) , x' , \alpha ' + \upsilon (\zeta_{(x,\alpha)} (s) , x'))$ with $\dot \upsilon (s,x) = - \frac{\Re(\psi_s (s,x))}{2\Im(\psi (s,x))}$. It remains to show that $\zeta_{(x,\alpha)} (s) = \frac{a'}{a} s$.

\begin{proposition}

The map $(S,X,\theta)$ sends a curve $\widetilde \gamma_{(x,\alpha)} (s) =  (s,x, \alpha + \tau (s,x))$ with $\tau_s (s,x) = - \frac{\Re(\phi_s (s,x))}{2\Im(\phi (s,x))}$ on a curve $\widetilde \delta_{(x' , \alpha ')} (s) = (\frac{a'}{a} s , x' , \alpha ' + \upsilon (\frac{a'}{a}s , x'))$ with $\upsilon_s (s,x) = - \frac{\Re(\psi_s (s,x))}{2\Im(\psi (s,x))}$. 

\end{proposition}

\begin{proof}
Again, the proof is very similar to the one of Proposition 2.2.4.. We consider the following two complex vector fields
\[ U := -\frac{i z \phi ' (\phi ^{-1} (t+i |z|^2))}{2|z|^2} Z \text{ and } \overline U := \frac{i \overline z \overline \phi ' (\phi ^{-1} (t+i |z|^2))}{2|z|^2} \zbar.\]
Then, using the same method as in Proposition 2.2.4., one may check that
\[|U (\Pi \circ f)| - |\overline U (\Pi \circ f)| = \frac{|f_1| |\phi ' (\phi ^{-1} (t+i|z|^2))|}{|z|} \left( |Zf_1| - |\zbar f_1 | \right). \]
Now, since $f_\ast \widetilde \rho_\phi \circ f = \widetilde \rho_\psi \circ f$ with $\widetilde \rho_\phi = \Pi^\ast \rho_\phi = \Pi^\ast \phi_\ast \rho_0$ and $\widetilde \rho_\psi = \Pi^\ast \rho_\psi = \Pi^\ast \psi_\ast \rho_0 '$, we have
\[ f_\ast \widetilde \rho_\phi \circ f = \frac{2|f_1|}{a\left(|U(\Pi \circ f)| - |\overline U (\Pi \circ f)| \right)} = \frac{2|f_1|}{a'|\psi ' (\psi ^{-1} \circ \Pi \circ f)|}.\]
Leading to 
\[|U(\Pi \circ f)| - |\overline U (\Pi \circ f)| = \frac{a'}{a} |\psi ' (\psi ^{-1} \circ \Pi \circ f)|.\]
Now, let $\widetilde \gamma_{(x,\alpha)} (s) =  (s,x, \alpha + \tau (s,x))$ with $\dot \tau (s,x) = - \frac{\Re(\phi_s (s,x))}{2\Im(\phi (s,x))}$ and $\gamma_{(x,\alpha)} (s) = \Psi_\phi (\widetilde \gamma_{(x,\alpha)}(s))$. Then  we have,
\[ \Pi \circ f \circ \widetilde \gamma_{(x,\alpha)} (s) = \psi (\zeta_{(x,\alpha)} (s) , x').\]
Thus, 
\[ |\dot{(\Pi \circ f \circ \gamma_{(x,\alpha)})}(s)| = |\psi ' (\zeta_{(x,\alpha)} (s) , x')||\dot \zeta_{(x,\alpha)}|.\]
But we also have
\begin{eqnarray*}
|\dot{(\Pi \circ f \circ \gamma_{(x,\alpha)})}(s)| & = & |U(\Pi \circ f) (\gamma_{(x,\alpha)} (s))| - |\overline U (\Pi \circ f)(\gamma_{(x,\alpha)}(s))|\\
& = & \frac{a'}{a} |\psi ' (\psi ^{-1} \circ \Pi \circ f \circ \gamma_{(x,\alpha)})(s)|\\
& = & \frac{a'}{a} |\psi ' (\zeta_{(x,\alpha)} (s) , x') |.
\end{eqnarray*}
So, $\zeta_{(x,\alpha)} (s)= \frac{a'}{a}s$ which ends the proof.

\end{proof}
Combining Propositions 3.2.5. and 3.2.8. is enough to prove Theorem 3.2.4..
\newline

We wish now to give two examples of the construction. The first one is between spherical annuli on the Heisenberg group and comes from \cite {BFP}, \cite {BFP2} where extremality and uniqueness was proved. Here, it is constructed using the holomorphic map $z \longmapsto e^z$. Applying Proposition 3.1.2. and Theorem 3.2.4., we are enabled to reconstruct the map and prove its uniqueness. The second example uses the translation $z \longmapsto z+i$. We find conditions on $a,b,a',b'$ for an extremal quasiconformal map to exist.

\begin{example}
1) Let us consider two half-annuli in $\mathbb H$ : $A_{a} := \{ w \in \mathbb H \ | \ 1<|w|<a^2 \}$ and $A_{a^k} := \{ w \in \mathbb H \ | \ 1<|w|<a^{2k} \}$ for $k<1$ and $a>1$. Then, $A_{a} = \phi (]0,2\ln (a)[) \times ]0,\pi[$ and $A_{a^k} = \phi (]0,2k\ln (a)[ \times ]0,\pi[)$, where $\phi (s,x) = e^{s+ix}$. Then, $\frac{|\phi ' (s,x)|^2}{\Im(\phi (s,x))^2} = \frac{1}{\sin ^2 (x)}$ is a function of $x$ only. Moreover, we denote by $\widetilde A_a = \Pi^{-1} (A_a) = \{ (z,t) \in \h \ | \ 1<\|(z,t)\|_\h <a \} \backslash \{ z = 0 \}$ and 
$\widetilde A_{a^k} = \Pi^{-1} (A_{a^k}) = \{ (z,t) \in \h \ | \ 1<\|(z,t)\|_h <a^k \} \backslash \{ z = 0 \}$ the spherical annuli in \h. 
The set $\mathcal F_\phi$ is here the set of quasiconformal map $f : \widetilde A_a \longmapsto A_{a^k}$ that extend homeomorphically on $\{ (z,t) \in \h \ | \ 1\le\|(z,t)\|_\h \le a \}$, sending $\{ \| (z,t)\|_\h = 1\}$ on $\{ \| (z,t)\|_\h = 1\}$, $\{ \| (z,t)\|_\h = a\}$ on $\{ \| (z,t)\|_\h = a^k\}$ and mapping the vertical line on itself. Finally, the family of curves considered here is the family of radial curves $\gamma_{(x,\alpha)} (s) = \left(\sqrt{e^s \sin x} e^{i(\alpha - \frac{\cot x}{2} s)} , e^s\cos x \right)$, has modulus $\pi ^2 \ln (a) ^{-3}$ with extremal density $\rho_\phi (z,t) = \frac{|z|}{\ln (a) \sqrt{t^2 + |z|^4}}$ for $\widetilde A_{a}$ and $\pi ^2 \ln (a^k) ^{-3}$ with extremal density $\rho_\psi (z,t) = \frac{|z|}{\ln (a^k) \sqrt{t^2 + |z|^4}}$ for $\widetilde A_{a^k}$(see Figure 2 next page).
\begin{figure}[!h]
\center
\includegraphics[width=12cm,height=6cm]{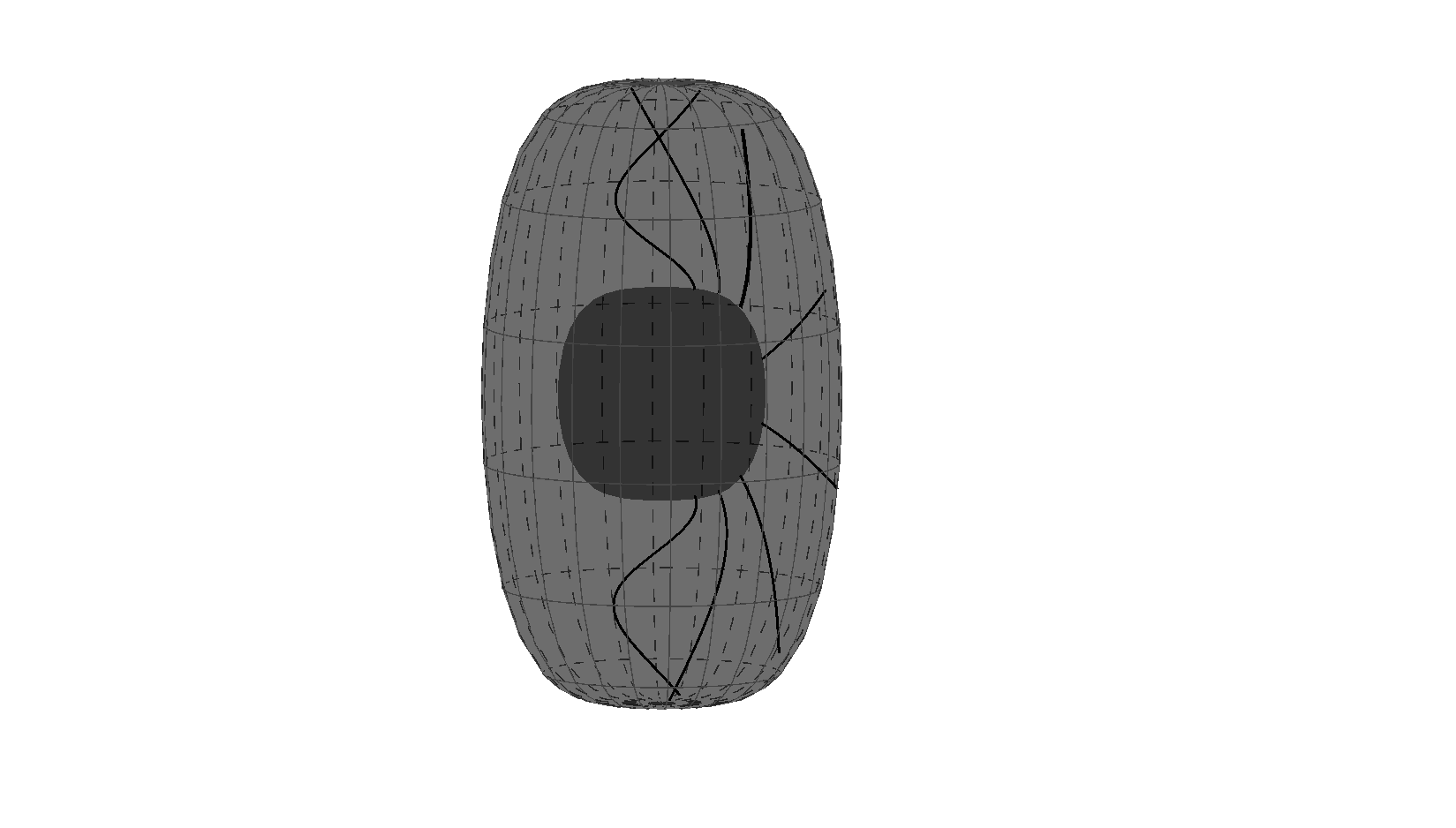}
\caption{Spherical annulus foliated by radial curves (foliation given by rotations around the vertical axis of drawn curves and the two pieces of the vertical axis itself).}
\end{figure}

According to Proposition 3.1.2. , if a lift up map $(S,X, \Theta)$ of $g_\varphi$ exists, it must verify the following :
\[ S(s,x,\theta) = ks, \ X(s,x,\theta) = \varphi (x)\]
\[\dot \varphi (x) \frac{\sin ^2 (x)}{\sin ^2 (\varphi (x))} = k^{-1}. \]
Thus, by solving the ordinary differential equation, we find for every $x$, $\varphi (x) = \cot ^{-1} (k^{-1} \cot (x) + D)$ where $D \in \R$. Moreover, $\dot \varphi (x) \ge k$ for every $x$. For $\varphi (x) = \cot ^{-1} \left(k^{-1}\cot (x) + D\right)$, this is equivalent to $k + 2D\cot (x) + kD^2 \le 1$ for every $x$. Which is possible if and only if $D = 0$. So, $\varphi (x) = \cot ^{-1} \left(k^{-1}\cot (x) \right)$ for every $x \in ]0,\pi[$. In particular, notice that $\varphi$ extends continuously in a homeomorphism from $[0,\pi]$ to $[0,\pi]$. By Proposition 3.1.2. again, we know now that we can find the function $\Theta$ to make $(S,X,\Theta)$ define a quasiconformal map between spherical annuli that minimises the mean distortion in $\mathcal F_\phi$ for the density $\widetilde \rho_\phi$. $\Theta (s,x,\theta) = \theta + h(s,x)$ where $h$ verifies 
\[2h_s (s,x) = 0 \text{ and } 2h_x (s,x) = \dot \varphi (x) - 1.\]
Thus, we find $h(s,x) = \frac{\varphi (x) - x}{2} + \theta_0$ for $\theta_0 \in \R$. Using $\Psi_\phi$ one is invited to check that in usual coordinates, it gives the rotations of the map 
\[\begin{array}{cccc}
f  : & \widetilde A_{a} & \longmapsto & \widetilde A_{a^k}\\
 & (z,t) & \longmapsto & \left( \sqrt k z \left(\frac{t-i|z|^2}{t-ik|z|^2} \right)^{\frac{1}{2}} |t+i|z|^2|^{\frac{k-1}{2}} , t \frac{| t+i|z|^2|^k}{|t+ik|z|^2|} \right)
\end{array}\]
which is the map studied in \cite{BFP}. 

For the uniqueness of that map (up to rotations) as a minimizer of the mean distortion in the class of all quasiconformal mappings between full spherical annuli (meaning between $\{p \in \h \ | \ 1< \|p\|_\h < a\}$ and $\{p \in \h \ | \ 1< \|p\|_\h < a^k\}$) sending homeomorphically boundary components on their corresponding ones, using Theorem 3.2.4., it is reduced to the verification of the fact that a minimizer has to send the vertical line homeomorphically on itself. 
\newline

2) Let us consider a subset of a cylinder $D_{r,R} := \{ (z,t) \in \h \ | \ 0<t<r, \ 1<|z|^2 < R+1 \}$. We are interested in the same minimisation problem as in Section 2 but this time between $D_{a,b}$ and $D_{a',b'}$. Meaning we consider a foliation of $D_{a,b}$ given by the subset of $\widetilde \Gamma_0$ given by curves that lie in $D_{a,b}$.
\begin{figure}[!h]
\center
\includegraphics[width=12cm,height=7cm]{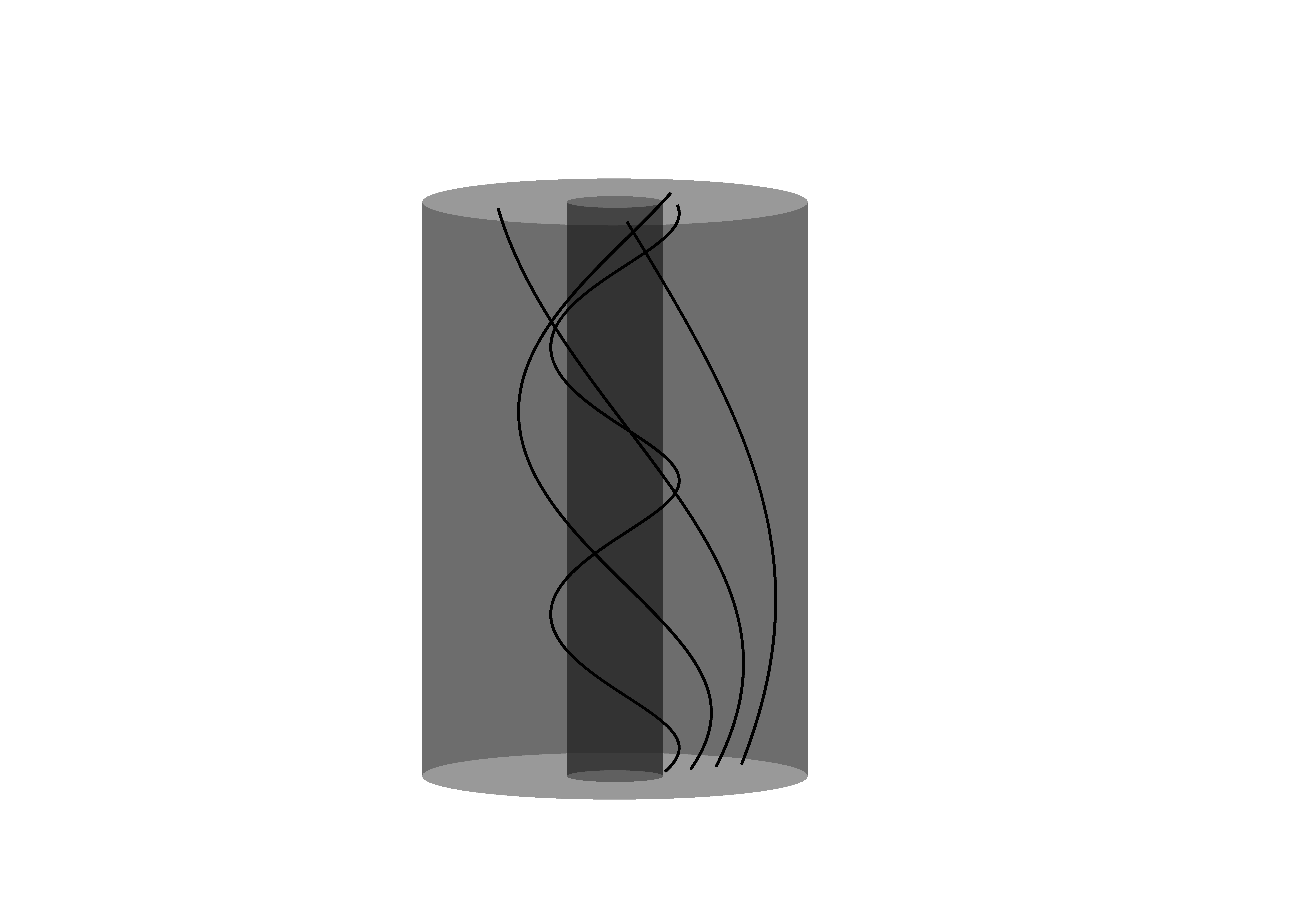}
\caption{$D_{a,b}$ foliated by a subset of $\widetilde \Gamma_0$ (foliation given by rotations around the vertical axis of drawn curves).}
\end{figure}

Those cylinders are simply lifts up by $\Pi$ of rectangles $\phi (R_{a,b})$ and $\phi (R_{a' , b'})$ for $\phi (w) = w+i$. According to Theorem 3.2.4., a minimising map 
$\widetilde g : D_{a,b} \longmapsto D_{a' , b'}$ for the mean distortion has to be constructed as a lift up map of one of the $g_\varphi$. We write the lift up map in coordinates $(s,x,\theta)$, in those coordinates, a minimizer is of the form $(\frac{a'}{a} s , \varphi (x) , \Theta (s,x,\theta))$. Now, according to Proposition 3.1.2., $\varphi [0,b] \longmapsto [0,b']$ must be a special function. It has to verify $\varphi (0) = 0$, $\varphi (b) = b'$, $\dot \varphi (x) \ge \frac{a'}{a}$ and finally, the ordinary differential equation 
\[\frac{a'}{a} \dot \varphi (x) \frac{(x+1)^2}{(\varphi (x) +1)^2} = 1 \]
whose solutions are $\varphi (x) = a'\frac{x + 1}{a+a'c(x+1)} - 1$ for $c\in \R$. From $\varphi (0) = 0$, we deduce that $c = 1-\frac{a}{a'}$. Now, in order that $\varphi (b) = b'$, then $a,b,a',b'$ must verify $\frac{a'b'}{ab} = \frac{1+b'}{1+b}$. In this condition, one may verify that $\widetilde g$ is the restriction to $C_{a,b+1} \backslash C_{a,1}$ of a map $\widetilde f_\alpha : C_{a,b+1} \longmapsto C_{a',b'+1}$ constructed in section 2, and $\widetilde f_\alpha$ maps the set $\{ (z,t) \in C_{a,b+1} \ | \ |z|=1 \}$ to $\{(z,t) \in C_{a',b'+1} \ | \ |z|=1 \}$. So, a minimizer of the mean distortion between $D_{a,b}$ and $D_{a' , b'}$ exists if and only if $\frac{a'b'}{ab} = \frac{1+b'}{1+b}$.

\end{example}

\end{document}